\documentclass[11pt,reqno]{amsart}

\usepackage{geometry}
\usepackage{amsmath,amsthm,amssymb,amsfonts}

\usepackage[utf8]{inputenc} 
\usepackage[T1]{fontenc}    
\usepackage[hidelinks]{hyperref} 
\usepackage[capitalise]{cleveref}

\usepackage{comment}
\usepackage{subcaption}
\usepackage{wrapfig}
\usepackage{bm}
\usepackage{url}            
\usepackage{booktabs}       
\usepackage{nicefrac}       
\usepackage{microtype}      

\usepackage{bbm}
\usepackage{cite}
\usepackage{psfrag}
\usepackage{cite}
\usepackage{color,soul}
\usepackage{breakcites}     
\usepackage{bigints}

\usepackage{enumitem}
\setlist{leftmargin=1.6em}

\newtheorem{theorem}{Theorem}[section]
\newtheorem{proposition}{Proposition}[section]
\newtheorem{assumption}{Assumption}[section]
\newtheorem{lemma}{Lemma} [section]
\newtheorem{corollary}{Corollary}[section]

\theoremstyle{definition}

\newtheorem{example}{Example}[section]
\theoremstyle{remark}
\newtheorem{remark}{Remark}[section]

\usepackage{graphicx}
\usepackage{xcolor}

\usepackage{algorithm}
\usepackage[noend]{algpseudocode}

\makeatletter
\def\subsubsection{\@startsection{subsubsection}{3}%
  \z@{.5\linespacing\@plus.7\linespacing}{-.5em}%
  {\normalfont\bfseries}}
\makeatother

\newcommand{\calB}[0]{\mathcal{B}}
\newcommand{\calC}[0]{\mathcal{C}}
\newcommand{\calF}[0]{\mathcal{F}}

\newcommand{\calH}[0]{\mathcal{H}}

\newcommand{\calP}[0]{\mathcal{P}}

\newcommand{\calX}[0]{\mathcal{X}}

\newcommand{\calY}[0]{\mathcal{Y}}

\newcommand{\supp}{\mathrm{spt}}
\newcommand{\E}[0]{\mathbb{E}}

\newcommand{\R}[0]{\mathbb{R}}

\newcommand{\Prob}[0]{\mathbb{P}}

\newcommand{\vertiii}[1]{{\left\vert\kern-0.25ex\left\vert\kern-0.25ex\left\vert #1 
    \right\vert\kern-0.25ex\right\vert\kern-0.25ex\right\vert}}

\newcommand{\NN}{\mathbb{N}}

\newcommand{\vasti}{\bBigg@{3.5 }}
\newcommand{\vast}{\bBigg@{4}}
\newcommand{\Vast}{\bBigg@{5}}
\newcommand{\Vastt}{\bBigg@{7}}

\makeatother

\newcommand{\be}{\begin{equation}}
\newcommand{\ee}{\end{equation}}
\newcommand{\ba}{\begin{align}}
\newcommand{\ea}{\end{align}}
\newcommand{\baa}{\begin{align*}}
\newcommand{\eaa}{\end{align*}}

\newif\ifedit

\begin{document}

\title{Large deviations for dynamical Schr\"{o}dinger problems}

\thanks{
K. Kato is partially supported by NSF grants DMS-1952306, DMS-2210368, and DMS-2413405.}

\date{First version: February 6, 2024. This version: \today}
\author[K. Kato]{Kengo Kato}
\address[K. Kato]{
Department of Statistics and Data Science, Cornell University.
}
\email{kk976@cornell.edu}

\begin{abstract}
We establish large deviations for dynamical Schr\"{o}dinger problems driven by perturbed Brownian motions when the noise parameter tends to zero.  Our results show that Schr\"{o}dinger bridges charge exponentially small masses outside the support of the limiting law that agrees with the optimal solution to the dynamical Monge-Kantorovich optimal transport problem. Our proofs build on mixture representations of Schr\"{o}dinger bridges and establishing exponential continuity of Brownian bridges with respect to the initial and terminal points. 
\end{abstract}



\maketitle

\section{Introduction}

\subsection{Overview}

The dynamical Schr\"{o}dinger problem \cite{follmer1988random,leonard2014survey} seeks to find the entropic projection of a reference path measure  (such as a Wiener measure) onto the space of path measures with given initial and terminal distributions. 
Originally motivated by physics, the problem has received increasing interest from other application domains such as statistics and machine learning; see \cite{bernton2019schr,pavon2021data,de2021diffusion,stromme2023sampling} and references therein. From a purely mathematical point of view, the time marginal flow, called entropic interpolation, provides a powerful technique for deriving functional inequalities and analysis of metric measure spaces \cite{boissard2014some,gentil2020entropic,gigli2021second}, making the dynamical Schr\"{o}dinger problem of intrinsic interest. Additionally, the static version of the Schr\"{o}dinger problem is equivalent to quadratic entropic optimal transport (EOT) \cite{nutz2021introduction}, the analysis of which has seen extensive research activities. This is in particular due to EOT admitting efficient computation via Sinkhorn's algorithm, which lends itself well to large-scale data analysis \cite{cuturi2013lightspeed,peyre2019computational}.

Schr\"{o}dinger problems can be interpreted as noisy counterparts of Monge-Kantorovich optimal transport (OT) problems. In particular,  \cite{mikami2004monge, mikami2008optimal, leonard2012schrodinger} studied 
the rigorous connection between the two problems, establishing convergence of optimal solutions for dynamical Schr\"{o}dinger problems (Schr\"{o}dinger bridges) toward the dynamical OT problem when the noise level tends to zero. In this work, we study local rates of convergence of Schr\"{o}dinger bridges toward the limiting law. Specifically, we establish large deviation principles (LDPs) for Schr\"{o}dinger bridges on a path space and characterize the rate function.

Our baseline setting goes as follows.
Let $\mu_0,\mu_1$ be Borel probability measures on $\R^d$ with finite second moments that will be fixed throughout. Let $E$ be the space of continuous maps $[0,1] \to \R^d$ endowed with the sup norm $\| \omega \|_E = \sup_{t \in [0,1]} |\omega (t)|$ for $\omega = (\omega (t))_{t \in [0,1]} \in E$ (we use $| \cdot |$ to denote the Euclidean norm). 
For a given $\varepsilon > 0$ (noise level), let $R^{\varepsilon}$ be the law, defined on the Borel $\sigma$-field of $E$, of $\xi + \sqrt{\varepsilon}W$, where $\xi \sim \mu_0$ and $W=(W(t))_{t \in [0,1]}$ is a standard Brownian motion starting at $0$ independent of $\xi$. For $s,t \in [0,1]$, we denote the projections at $t$ and $(s,t)$ as $e_t$ and $e_{st}$, respectively, i.e., $e_t(\omega) =\omega (t)$ and $e_{st} (\omega) = (\omega (s),\omega(t))$ for $\omega \in E$.
For a given Borel probability measure $P$ on $E$, denote $P_t = P \circ e_t^{-1}$ and $P_{st} = P \circ e_{st}^{-1}$.
Given two endpoint marginals $\mu_0,\mu_1$ and a reference measure $R^{\varepsilon}$, the dynamical Schr\"{o}dinger problem reads as
\begin{equation}
\min_{P: P_0=\mu_0,P_1=\mu_1} \calH(P | R^\varepsilon), \label{eq: dynamical Schrodinger}
\end{equation}
where $\calH(\cdot \mid \cdot)$ denotes the relative entropy (see Section \ref{sec: notation} for the formal definition). 
Provided $\mu_1$ has finite entropy relative to the Lebesgue measure (cf. Remark \ref{rem: entropy}), the problem (\ref{eq: dynamical Schrodinger}) admits a unique optimal solution $P^{\varepsilon}$, called the \textit{Schr\"{o}dinger bridge}. The solution $P^{\varepsilon}$ is given by a mixture of Brownian bridges against a (unique) optimal solution $\pi_\varepsilon$ to the static Schr\"{o}dinger problem 
\begin{equation}
\min_{\pi \in \Pi(\mu_0,\mu_1)} \calH(\pi | R_{01}^\varepsilon),
\label{eq: static Schrodinger one}
\end{equation}
where $\Pi(\mu_0,\mu_1)$ is the set of couplings with marginals $\mu_0$ and $\mu_1$.  The zero-noise limit ($\varepsilon \downarrow 0$) of (\ref{eq: static Schrodinger one}) corresponds to the OT problem with quadratic cost $c(x,y) = |x-y|^2/2$,
\begin{equation}
\min_{\pi \in \Pi(\mu_0,\mu_1)} \int c \, d\pi,
\label{eq: OT problem}
\end{equation}
which admits a unique optimal solution (OT plan) $\pi_o$ (as $\mu_1$ is assumed to be absolutely continuous; \cite{brenier1991polar}). 

In his influential work \cite{mikami2004monge}, Mikami proved, under an additional assumption that $\mu_0$ is absolutely continuous, that
$P^{\varepsilon}$ converges weakly to the law $P^o$ of the geodesic path connecting two random endpoints following $\pi_o$, $t \mapsto \sigma^{\xi_0,\xi_1}(t)$ for $\sigma^{xy}(t) = (1-t)x+ty$ and $(\xi_0,\xi_1) \sim \pi_o$, i.e., $P^o = \int \delta_{\sigma^{xy}} \, d\pi_o (x,y)$ with $\delta_{\cdot}$ denoting the Dirac delta.\footnote{\cite{mikami2004monge} indeed proved convergence w.r.t. Wasserstein $W_2$ distance.} The limiting law $P^o$ can be characterized as an optimal solution to the dynamical OT problem
\[
\min_{P: P_0=\mu_0, P_1=\mu_1} \int \left ( \frac{1}{2}\int_{0}^1 |\dot{\omega}(t)|^2 \, dt \right ) \, dP(\omega),
\]
where $\dot{\omega}(t)$ denotes the time derivative of $\omega$ and $\int_{0}^1 |\dot{\omega}(t)|^2 \, dt =\infty$ if $\omega$ is not absolutely continuous \cite{leonard2012schrodinger}.
The marginal laws of the limiting process give rise to a constant-speed geodesic (displacement interpolation; \cite{mccann1997convexity}) in the Wasserstein space connecting $\mu_0$ and $\mu_1$. 

Our main large deviation results establish that\footnote{See Sections \ref{sec: notation} and \ref{sec: prelim} for notations and definitions.}, under regularity conditions, for any sequence $\varepsilon_k \downarrow 0$, 
the Schr\"{o}dinger bridges $P^{\varepsilon_k}$ satisfy an LDP with rate function $I(h) = \int_0^1(|\dot{h}(t)|^2/2) \, dt - \psi^c(h(0))  - \psi(h(1))$, where $\psi$ is an OT (or Kantorovich) potential from $\mu_1$ to $\mu_0$ and $\psi^c$ is its $c$-transform (the rate function $I$ is set to $\infty$ if $h(0)$ or $h(1)$ is outside the support of $\mu_0$ or $\mu_1$, respectively).  Very roughly, this means  $P^{\varepsilon_k}(A) \approx e^{-\varepsilon_k^{-1} \inf_{h \in A}I(h)}$ for large $k$. The rate function $I(h)$ vanishes as soon as $h \in \Sigma_{\pi_0} :=  \{ \sigma^{xy} :(x,y) \in \supp(\pi_o) \}$, which agrees with the support of $P^o$, but $I(h)$ is  positive outside $\Sigma_{\pi_0}$ in many cases. Effectively, our result implies that the Schr\"{o}dinger bridges $P^{\varepsilon}$ charge exponentially small masses outside the support of the limiting law $P^o$. Precisely, we establish a weak-type LDP under uniqueness of OT potentials, which allows for marginals with unbounded supports, but induces a full LDP when $\mu_0,\mu_1$ are compactly supported.  

The proof of the main theorem relies on the expression of $P^{\varepsilon}$  as a $\pi_\varepsilon$-mixture of Brownian bridges. The main ingredient of the proof is \textit{exponential continuity}\cite{dinwoodie1992large} of Brownian bridges, i.e., establishing large deviation upper and lower bounds for Brownian bridges when the locations of initial and terminal points vary with the noise level. Note that an LDP for Brownian bridges with \textit{fixed} initial and terminal points was derived in \cite{hsu1990brownian}, but Hsu's proof, which relies on transition density estimates, seems difficult to adapt to establishing the exponential continuity. Instead, we use techniques from abstract Wiener spaces (cf. Chapter 8 in \cite{stroock2010probability}) to establish the said result. Given the exponential continuity, the main theorem follows from combining the large deviation results for $\pi_\varepsilon$ established in \cite{bernton2021entropic}. For the compact support case, we provide a more direct proof of the full LDP using the representation of $P^{\varepsilon}$ as  an integral of a $(\mu_0 \otimes \mu_1)$-mixture of Brownian bridges.  The proof first shows an LDP for the $(\mu_0 \otimes \mu_1)$-mixture of Brownian bridges, and then establishes the full LDP by adapting the (Laplace-)Varadhan lemma (cf. Theorem 4.4.2 in \cite{dembo2009large}) and using convergence of EOT (or Schr\"{o}dinger) potentials. {The alterative proof can be easily adapted to establish an LDP for the dynamical Schr\"{o}dinger problem with Langevin diffusion as a reference measure when two marginals are compactly supported; cf. Remark \ref{rem: langevin} ahead.}

\subsection{Literature review}
The literature related to this paper is broad, so we confine ourselves to the references directly related to our work.
The most closely related are \cite{bernton2021entropic,nutz2021entropic}, which established large deviations for static Schr\"{o}dinger problems in fairly general settings, allowing for marginals on a general Polish space and general continuous costs, and our proofs use several results from their work. \cite{bernton2021entropic} derived a weak LDP for EOT via a novel cyclical invariance characterization of EOT plans, while \cite{nutz2021entropic} built on convergence of EOT potentials. 

The connection between Schr\"{o}dinger and OT problems has been one of the central problems in the OT literature. We focus here on convergence of Schr\"{o}dinger problems. The pioneering works in this direction are \cite{mikami2004monge,mikami2008optimal,leonard2012schrodinger}. Mikami's proof in \cite{mikami2004monge} relies on the fact that the Schr\"{o}dinger bridge $P^{\varepsilon}$ corresponds to a weak solution of a certain stochastic differential equation (SDE) with diffusion component $\sqrt{\varepsilon} \, dW(t)$, the special case of which is often referred to as the \textit{F\"{o}llmer process} \cite{lehec2013representation,mikulincer2021brownian}; see Remark \ref{rem: follmer} below. The drift function of the said SDE being dependent on $\varepsilon$ in a nontrivial way (among others) makes the problem of large deviations for dynamical Schr\"{o}dinger problems fall outside the realm of the Freidlin-Wentzell theory (cf. Chapter 5 in \cite{dembo2009large}). On the other hand, L\'{e}onard's proof in \cite{leonard2012schrodinger} relies on the variational representation of the relative entropy and convex analysis techniques to establish $\Gamma$-convergence of the Sch\"{o}dinger objective functions, which yields convergence of the optimal solutions. Arguably, recent interest in EOT (static Schr\"{o}dinger problem) stems from the fact that EOT provides an efficient computational means for unregularized OT \cite{cuturi2013lightspeed,peyre2019computational}.
From this perspective, extensive research has been done on convergence and speed of convergence of EOT costs, potentials, plans, and maps toward those of unregularized OT \cite{carlier2017convergence,pal2019difference,chizat2020faster,conforti2021formula,nutz2021entropic,pooladian2021entropic,altschuler2022asymptotics,carlier2023convergence}.

To the best of the author's knowledge, this is the first paper to establish large deviations for dynamical Schr\"{o}dinger problems.
As noted in the beginning, the dynamical aspect of the Schr\"{o}dinger bridge has received increasing interest from application domains, which calls for further research on this subject. Our results contribute to the rigorous understanding of the connection between the dynamical Schr\"{o}dinger and OT problems in the small-noise regime. From a technical perspective, our use of mixture representations to explore large deviations on path spaces might be applied to other problems.
Finally, in this work, we focus on the Wiener reference measure that corresponds to the quadratic OT problem. Arguably, this setting would be the most basic. Extending our large deviation results to the dynamical problem in abstract metric spaces \cite{monsaingeon2023dynamical} would be of interest, but beyond the scope of this paper.

\subsection{Organization}
 The rest of the paper is organized as follows. Section \ref{sec: prelim} contains background on EOT, Schr\"{o}dinger, and OT problems, and Section \ref{sec: main} presents the main results. All the proofs  are gathered in Section \ref{sec: proof}.

\subsection{Notations and definitions}
\label{sec: notation}
Let $x \cdot y$ denote the Euclidean inner product for $x,y \in \R^d$.
For $x,y \in \R^d$ and a Borel probability measure $P$ on $E$, let $P^{xy}$ denote the (regular) conditional law of $X$ given $(X(0),X(1)) = (x,y)$ for $X=(X(t))_{t \in [0,1]} \sim P$. For a set $A$, let $\iota_{A}(x) = 0$ if $x\in A$ and $=\infty$ if $x \notin A$.  On a metric space $M$, let $B_M(x,r)$ denote the open ball in $M$ with center $x$ and radius $r$. For a Borel probability measure $\mu$ on a metric space, its support is denoted by $\supp(\mu)$. For probability measures $\alpha,\beta$ on a common measurable space, $\calH(\alpha | \beta)$ is the relative entropy defined as 
\[
\calH(\alpha | \beta) :=  
\begin{cases}
  \int \log \frac{d\alpha}{d\beta} \, d\alpha & \text{if} \ \alpha \ll \beta, \\
  \infty & \text{otherwise}.
\end{cases}
\]

A lower semicontinuous function $I: M \to [0,\infty]$ defined on a metric space $M$ is called a \textit{rate function}. The rate function $I$ is called \textit{good} if all level sets $\{ x : I(x) \le \alpha \}$ for $\alpha \in [0,\infty)$ are compact. 
Given a sequence of positive reals $a_k \to \infty$, a sequence of Borel probability measures $\{ P_k \}_{k \in \NN}$ on $M$ satisfies a weak \textit{large deviation principle} (LDP) with \textit{speed} $a_k$ and rate function $I$,  if 
\begin{enumerate}
    \item[(i)] for every open set $A \subset M$, 
    \[
    \liminf_{k \to \infty}a_k^{-1} \log P_k(A) \ge - \inf_{x \in A} I(x),
    \]
    and
    \item[(ii)] for every compact set $A \subset M$,
\[
\limsup_{k \to \infty} a_k^{-1} \log P_k(A) \le - \inf_{x \in A}I(x). 
\]
\end{enumerate}
If condition (ii) holds for every closed set $A \subset M$, then we say that $\{ P_k \}_{k \in \NN}$ satisfies a (full) LDP. 
We refer the reader to \cite{dembo2009large} as an excellent reference on large deviations.

\section{Preliminaries} 
\label{sec: prelim}

\subsection{From EOT to Schr\"{o}dinger problems}
\label{sec: EOT}
We first review EOT and its connection to the Schr\"{o}dinger problems, which will play a key role in the proofs of the main results. Proofs of the results below can be found in \cite{leonard2014survey} or \cite{nutz2021entropic}. Throughout, we set $\calX = \supp(\mu_0)$ and $\calY = \supp(\mu_1)$. 

Given marginals $\mu_0,\mu_1$, the EOT problem for quadratic cost $c(x,y) = |x-y|^2/2$ reads as
\begin{equation}
\min_{\pi \in \Pi (\mu_0,\mu_1)} \int c \, d\pi + \varepsilon \calH(\pi | \mu_0 \otimes \mu_1) = \min_{\pi \in \Pi (\mu_0,\mu_1)} \varepsilon \left ( \int (c/\varepsilon) \, d\pi + \calH(\pi | \mu_0 \otimes \mu_1) \right). \label{eq: EOT}
\end{equation}
Setting $d\nu_\varepsilon=Z_\varepsilon^{-1} e^{-c/\varepsilon} \, d(\mu_0 \otimes \mu_1) $ with $Z_{\varepsilon}= \int e^{-c/\varepsilon} \, d(\mu_0 \otimes \mu_1)$, we have
\[
\int (c/\varepsilon) \, d\pi + \calH(\pi | \mu_0 \otimes \mu_1) 
=\calH(\pi  | \nu_{\varepsilon}) - \log Z_{\varepsilon},
\]
which implies that \eqref{eq: EOT} is equivalent to the following static Schr\"{o}dinger problem
\begin{equation}
\min_{\pi \in \Pi(\mu_0,\mu_1)} \calH(\pi | \nu_{\varepsilon}).
\label{eq: static Schrodinger}
\end{equation}
Recall that $\Pi (\mu_0,\mu_1)$ is compact for the weak topology. 
Since $\pi \mapsto \calH(\pi | \nu_\varepsilon)$ is lower semicontinuous with respect to (w.r.t.) the weak topology (which follows from the variational representation of the relative entropy) and strictly convex on the set of $\pi$ such  that $\calH(\pi | \nu_\varepsilon)$ is finite (which follows from strict convexity of $x \mapsto x\log x$), the problem \eqref{eq: static Schrodinger} admits a unique optimal solution $\pi_{\varepsilon}$, provided $\calH(\pi | \nu_{\varepsilon}) < \infty$ for some $\pi \in \Pi (\mu_0,\mu_1)$. Since $\mu_0$ and $\mu_1$ have finite second moments, we have $\calH(\mu_0 \otimes \mu_1 | \nu_{\varepsilon}) < \infty$.
We will call $\pi_{\varepsilon}$ the \textit{EOT plan}. 

The EOT plan has a density w.r.t.  $\mu_0 \otimes \mu_1$ given by
\[
d\pi_{\varepsilon} (x,y)= e^{(\varphi_{\varepsilon} (x)+ \psi_{\varepsilon} (y)- c(x,y))/\varepsilon} \, d(\mu_0 \otimes \mu_1)(x,y), 
\]
where $\varphi_{\varepsilon} \in L^1(\mu_0)$ and $\psi_{\varepsilon} \in L^1(\mu_1) $ are \textit{EOT potentials}  satisfying the Schr\"{o}dinger system
\begin{equation}
\begin{cases}
&\int e^{(\varphi_{\varepsilon} (x) +  \psi_{\varepsilon} (y) - c(x,y))/\varepsilon} \, d\mu_1 (y)= 1,  \ \text{$\mu_0$-a.e. $x$}, \\
&\int e^{(\varphi_{\varepsilon}(x) +  \psi_{\varepsilon}(y) - c(x,y))/\varepsilon} \, d\mu_0(x) = 1, \ \text{$\mu_1$-a.e. $y$}. 
\end{cases}
\label{eq: S system}
\end{equation}
EOT potentials are a.s. unique up to additive constants, i.e., if $(\tilde{\varphi}_{\varepsilon},\tilde{\psi}_{\varepsilon})$ is another pair of EOT potentials, then there exists a constant $a \in \R$ such that $\tilde{\varphi}_{\varepsilon} = \varphi_{\varepsilon} + a$ $\mu_0$-a.e. and $\tilde{\psi}_{\varepsilon} = \psi_{\varepsilon} - a$ $\mu_1$-a.e. 
In many cases (e.g. as soon as $\mu_0,\mu_1$ are sub-Gaussian), one can choose versions of (finite) EOT potentials for which the Schr\"{o}dinger system (\ref{eq: S system}) holds for all $x \in \calX$ and $y \in \calY$ (in fact for all $x \in \R^d$ and $y \in \R^d$); see Proposition 6 in \cite{mena2019statistical}. Whenever possible, we always choose such versions of EOT potentials.

To link EOT to the original static Schr\"{o}dinger problem (\ref{eq: static Schrodinger one}), we make the following assumption. 
\begin{assumption}
$\mu_1 \ll dy$ and $\calH(\mu_1 | dy) < \infty$. 
\label{asp: finite entropy}
\end{assumption}

\begin{remark}[On the relative entropy $\calH(\mu_1|dy)$]
\label{rem: entropy}
Here, as in Appendix A in \cite{leonard2014survey}, we define the relative entropy $\calH(\mu_1|dy)$ against the Lebesgue measure $dy$ given by
\[
\calH(\mu_1 | dy) := \int \log (\rho/\mathfrak{g}) \, d\mu_1 + \int (\log \mathfrak{g}) \, d\mu_1 \in (-\infty,\infty]
\]
where $\rho = d\mu_1/dy$ and $\mathfrak{g}$ is the standard Gaussian density on $\R^d$.
\end{remark}

The reference measure $R_{01}^{\varepsilon} = R^{\varepsilon} \circ e_{01}^{-1}$ for (\ref{eq: static Schrodinger one}) has a density w.r.t. $dy d\mu_0(x)$ given by 
$
dR_{01}^\varepsilon(x,y) = (2\pi\varepsilon)^{-d/2} e^{-c(x,y)/\varepsilon} \, dy d\mu_0(x)$, 
so $\nu_{\varepsilon}$ is absolutely continuous w.r.t. $R_{01}^\varepsilon$ with density $d\nu_{\varepsilon}(x,y) 
=(2\pi\varepsilon)^{d/2} Z_{\varepsilon}^{-1} \rho(y) \, dR^{\varepsilon}_{01}(x,y)$.
Hence,
\[
\calH(\pi |\nu_{\varepsilon}) = \calH(\pi | R^{\varepsilon}_{01}) -\frac{d}{2}\log (2\pi\varepsilon) + \log Z_{\varepsilon} -\calH(\mu_1 | dy),
\]
and the unique optimal solution to (\ref{eq: static Schrodinger one}) is given by  $\pi_{\varepsilon}$. 

Going back to the dynamical Schr\"{o}dinger problem \eqref{eq: dynamical Schrodinger}, by the chain rule for the relative entropy, we have
\[
\calH(P | R^{\varepsilon}) = \calH(P_{01} | R_{01}^{\varepsilon}) +  \int \calH(P^{xy} | R^{\varepsilon,xy}) \, dP_{01}(x,y),
\]
which is minimized by taking $P^{xy} = R^{\varepsilon,xy}$ and $P_{01}=\pi_\varepsilon$, i.e.,
\begin{align}
P^{\varepsilon}(\cdot) &= \int R^{\varepsilon,xy}(\cdot) \, d\pi_{\varepsilon}(x,y) \label{eq: mixture} \\
&= \int e^{(\varphi_{\varepsilon}(x) + \psi_{\varepsilon}(y) - c(x,y))/\varepsilon} R^{\varepsilon,xy} (\cdot) \, d (\mu_0 \otimes \mu_1)(x,y). \notag
\end{align}
Alternatively, setting 
$
\bar{R}^{\varepsilon} = \int R^{\varepsilon,xy} \, d(\mu_0 \otimes \mu_1), 
$
which is a $(\mu_0 \otimes \mu_1)$-mixture of Brownian bridges, $P^\varepsilon$ has a density w.r.t. $\bar{R}^\varepsilon$ given by
\begin{equation}
\frac{dP^{\varepsilon}}{d\bar{R}^{\varepsilon}} (\omega) = e^{-\phi_{\varepsilon}(\omega(0),\omega(1))/\varepsilon}, \ \omega = (\omega(t))_{t \in [0,1]} \in E,
\label{eq: S bridge density}
\end{equation}
where $\phi_{\varepsilon}: \calX \times \calY \to \R$ is a function defined by
\[
\phi_{\varepsilon}(x,y) =  c(x,y) - \varphi_{\varepsilon}(x) - \psi_{\varepsilon}(y). 
\]
To see this, for $X = (X(t))_{t \in [0,1]} \sim \bar{R}^{\varepsilon}$ and every Borel set $A \subset E$, 
\[
\begin{split}
\E\left [ \mathbbm{1}_{A} (X) e^{-\phi_{\varepsilon}(X(0),X(1))/\varepsilon} \right ] 
&=\E \left [ \Prob \big( X \in A \mid X(0),X(1)  \big)  e^{-\phi_{\varepsilon}(X(0),X(1))/\varepsilon}\right] \\
&=\E \left [ R^{\varepsilon, (X(0),X(1))}(A) e^{-\phi_{\varepsilon}(X(0),X(1))/\varepsilon} \right ] = P^{\varepsilon}(A), 
\end{split}
\]
where we used $(X(0),X(1)) \sim \mu_0 \otimes \mu_1$. 

\begin{remark}[On Assumption \ref{asp: finite entropy}]
Assumption \ref{asp: finite entropy} is unavoidable to ensure the problem (\ref{eq: static Schrodinger one}) to have a unique optimal solution. On the other hand, the initial distribution $\mu_0$ need not be absolutely continuous, e.g., $\mu_0$ can be discrete. 
\end{remark}

\begin{remark}[Connnection to F\"{o}llmer process]
\label{rem: follmer}
The Schr\"{o}dinder bridge $P^{\varepsilon}$ corresponds to the law of a weak solution to a certain SDE, the special case of which is often referred to as the F\"{o}llmer process. Let $\calB(E)$ be the Borel $\sigma$-field on $E$.  
Equip $(E,\calB(E),R^{\varepsilon})$ with the canonical filtration (augmented, if necessary), and denote by $X=(X(t))_{t \in [0,1]}$ the canonical process, i.e., $X(t,\omega) = \omega(t)$ for $\omega = (\omega(t))_{t \in [0,1]} \in E$. Under $R^\varepsilon$, $W  = \varepsilon^{-1/2} (X-X(0))$ is a standard Brownian motion starting at $0$. Assuming  $\rho=d\mu_1/dy$ is smooth and everywhere positive, we set
\[
\tilde{\psi}_\varepsilon (y)  := \varepsilon \big ( (d/2) \log (2\pi\varepsilon) + \log \rho(y) \big) + \psi_{\varepsilon}(y). 
\]
With this notation, it is seen that
\[
P^\varepsilon (\cdot)= \int e^{(\varphi_{\varepsilon}(x)+\tilde{\psi}_\varepsilon(y))/\varepsilon} R^{\varepsilon,xy}(\cdot) \, dR_{01}^{\varepsilon}(x,y),
\]
which implies (cf. the preceding argument) 
\[
\frac{dP^{\varepsilon}}{dR^{\varepsilon}}  = e^{(\varphi_{\varepsilon}(X(0))+\tilde{\psi}_\varepsilon(X(1)))/\varepsilon}.
\]
Denote 
\[
\mathfrak{h}_\varepsilon (t,y) := 
\begin{cases}
\big(2\pi\varepsilon (1-t)\big)^{-d/2} \int e^{-\frac{1}{\varepsilon} \left ( \frac{c(y,y')}{1-t}-\tilde{\psi}_{\varepsilon}(y') \right)} \, dy' & \text{if} \ t \in [0,1), \\
e^{\tilde{\psi}_{\varepsilon}(y)/\varepsilon} & \text{if} \ t = 1,
\end{cases}
\]
which satisfies $(\partial_t + \varepsilon \Delta_y/2) \mathfrak{h}_\varepsilon  = 0$ under regularity conditions (cf. heat equation). Applying Ito's formula (cf. Theorem 3.3.6 in \cite{karatzas1998}), one has
\[
\begin{split}
\log \mathfrak{h}_{\varepsilon}(1,X(1)) &= \underbrace{\log \mathfrak{h}_{\varepsilon}(0,X(0))}_{=-\varphi_{\varepsilon}(X(0))/\varepsilon} +  \frac{1}{\sqrt{\varepsilon}} \int_0^1 b_\varepsilon (t,X(t)) \cdot \, dW(t) -\frac{1}{2\varepsilon} \int |b_\varepsilon(t,X(t)) |^2 \, dt, 
\end{split}
\]
where we define $b_\varepsilon (t,y) = \varepsilon \nabla_y \log \mathfrak{h}_\varepsilon (t,y)$. 
We conclude that $R^\varepsilon$-a.s., 
\[
\frac{dP^{\varepsilon}}{dR^{\varepsilon}}  = \exp \left \{ \frac{1}{\sqrt{\varepsilon}} \int_0^1  b_\varepsilon (t,X(t)) \cdot \, dW(t) - \frac{1}{2\varepsilon} \int |  b_\varepsilon(t,X(t)) |^2 \, dt \right \}.
\]
By Girsanov's theorem, under $P^{\varepsilon}$, the new process $\tilde{W} = W - \varepsilon^{-1/2} \int_0^\cdot b_\varepsilon (t,X(t)) \, dt$ is a standard Brownian motion, and the process $X$ solves the following SDE
\begin{equation}
dX(t) = b_{\varepsilon}(t,X(t)) \, dt + \sqrt{\varepsilon} \, d\tilde{W}(t), \quad X(0) \sim \mu_0. 
\label{eq: follmer}
\end{equation}
See, e.g., Proposition 5.3.6 in \cite{karatzas1998}; see also \cite{dai1991stochastic}. When $\varepsilon=1$ and $\mu_0 = \delta_0$, one can take $\varphi_{1}(x) = 0$ and $\psi_{1}(y) = |y|^2/2$, so $e^{\tilde{\psi}_{1}}$ is the density of $\mu_1$ w.r.t. the standard Gaussian. Hence, the SDE in (\ref{eq: follmer}) corresponds to the F\"{o}llmer process in \cite{lehec2013representation} (see equation (12)) and \cite{mikulincer2021brownian}.
Abusing terminology, we call $P^{\varepsilon}$ with $\varepsilon > 0$ and $\mu_0 = \delta_0$ a (perturbed) F\"{o}llmer process. 
\end{remark}


\subsection{OT potentials}
\label{sec: Kantorovich}
The rate function for Schr\"{o}dinger bridges involves OT potentials. For duality theory of OT, we refer the reader to \cite{ambrosio2008gradient,villani2009optimal,santambrogio15}. 
The OT problem (\ref{eq: OT problem}) admits a dual problem that reads as
\begin{equation}
\max_{\substack{(\varphi,\psi) \in L^1(\mu_0) \times L^1(\mu_1) \\ \varphi + \psi \le c}} \int \varphi \, d\mu_0 + \int \psi \, d\mu_1. 
\label{eq: dual}
\end{equation}
By restricting to the respective support, it is without loss of generality to assume that $\varphi$ and $\psi$ are functions defined on $\calX$ and $\calY$, respectively.
One of $\varphi$ and $\psi$ can be replaced with the \textit{$c$-transform} of the other. Recall that the $c$-transform of $\psi: \calY \to [-\infty,\infty)$ with $\psi \not \equiv -\infty$ is a function $\psi^c: \calX \to [-\infty,\infty)$ defined by
\[
\psi^c (x) := \inf_{y \in \calY} \big \{ c(x,y) - \psi (y) \big \}, \ x \in \calX. 
\]
The $c$-transform of $\varphi: \calX \to [-\infty,\infty)$ with $\varphi \not \equiv -\infty$ is defined analogously. The dual problem (\ref{eq: dual}) then reduces to 
\begin{equation}
\max_{\psi \in L^1(\mu_1)} \int \psi^c \, d\mu_0 + \int \psi \, d\mu_1,
\label{eq: semidual}
\end{equation}
whose maximum is attained at some $c$-concave function $\psi \in L^1(\mu_1)$ with $\psi^c \in L^1(\mu_0)$ (a function on $\calY$ is called \textit{$c$-concave} if it is the $c$-transform of a function on $\calX$); see, e.g., Theorem 5.9 in \cite{villani2009optimal} or Theorem 6.1.5 in \cite{ambrosio2008gradient}. We call such $\psi$ an \textit{OT potential} from $\mu_1$ to $\mu_0$. An OT potential from $\mu_0$ to $\mu_1$ is defined analogously.

For any OT potential $\psi$ and any OT plan $\pi$, the support of $\pi$ is contained in the \textit{$c$-superdifferential} $\partial^c \psi$ of $\psi$,
\[
\partial^c \psi := \big\{ (x,y) : \psi^c(x) + \psi(y) = c(x,y) \big\}.
\]
Indeed, $\partial^c \psi$ is a closed set (as $c$-concave functions are upper semicontinuous) on which $\pi$ has full measure by duality, so $\supp(\pi)\subset \partial^c \psi$. In particular, for $(x,y) \in \supp(\pi)$, $\psi^c(x)$ and $\psi(y)$ are finite. 

Observe that Assumption \ref{asp: finite entropy} ensures that
\[
\text{the OT problem (\ref{eq: OT problem}) admits a unique OT plan $\pi_o$.}
\]
Let $\calX_o$ and $\calY_o$ denote the projections of $\supp(\pi_o)$ onto $\calX$ and $\calY$, respectively, i.e., 
\[
\calX_o = \{ x : (x,y) \in \supp(\pi_o) \ \text{for some} \ y\},
\]
and $\calY_o$ is defined analogously. As $\pi_o$ is a coupling for $\mu_0$ and $\mu_1$, the sets $\calX_o$ and $\calY_o$ have full $\mu_0$- and $\mu_1$-measure, respectively. As in \cite{bernton2021entropic}, we assume uniqueness of OT potentials (from $\mu_1$ to $\mu_0$) on $\calY_o$ to derive our large deviation results.  

\begin{assumption}
The dual problem (\ref{eq: semidual}) admits a unique OT potential $\psi$ on $\calY_o$, i.e., if $\tilde{\psi}$ is another OT potential, then $\psi - \tilde{\psi}$ is constant on $\calY_o$. 
\label{asp: uniqueness}
\end{assumption}

Appendix B in \cite{bernton2021entropic} and \cite{staudt2022uniqueness} provide various sufficient conditions under which uniqueness of OT potential holds. 
For example, Assumption \ref{asp: uniqueness} holds under each of the following cases. 

\begin{enumerate}
\item[(A)] (\cite{santambrogio15}, Theorem 7.18) \ $\calX$ and $\calY$ are compact and one of them agrees with the closure of a connected open set. 
\item[(B)] (\cite{bernton2021entropic}, Proposition B.2) \ The interior $\mathrm{int}(\calY)$ is connected, $\mu_1$ is absolutely continuous with positive Lebesgue density  on $\mathrm{int}(\calY)$, and $\mu_1(\partial \calY)=0$.
\end{enumerate}

Case (A) does not require $\mu_0$ or $\mu_1$ to have a Lebesgue density (although Assumption \ref{asp: finite entropy} requires $\mu_1 \ll dy$). 
We provide a self-contained proof of Case (A) in Lemma \ref{lem: uniqueness} for completeness. Case (B) imposes no restrictions on $\mu_0$, so it allows $\mu_0$ to be discrete.

\begin{remark}
Often, regularity conditions are imposed on the input measure $\mu_0$ to ensure uniqueness or regularity of OT potentials from $\mu_0$ to $\mu_1$. For the (static) EOT case, the role of $\mu_0$ and $\mu_1$ is symmetric, so it is without loss of generality to focus on the forward ($\mu_0 \to \mu_1$) case. However, in our dynamical setting, the roles of $\mu_0$ and $\mu_1$ are asymmetric because of Assumption \ref{asp: finite entropy}.
Since Assumption \ref{asp: finite entropy} already imposes absolute continuity on $\mu_1$, we treat OT potentials for the backward direction ($\mu_1 \to \mu_0$), contrary to the convention in the literature. 
\end{remark}

\section{Main results}
\label{sec: main}
We first recall  weak convergence  of $P^\varepsilon$ toward $P^o =  \int \delta_{\sigma^{xy}} \, d\pi_o(x,y)$ with $\sigma^{xy}(t) = (1-t)x+ty$. {Recall that the cost function is $c(x,y) = |x-y|^2/2$.}

\begin{proposition}
\label{prop: weak conv}
Under Assumption \ref{asp: finite entropy}, $P^\varepsilon \to P^o$ weakly as $\varepsilon \downarrow 0$. The support of $P^o$ agrees with $\Sigma_{\pi_o} := \{ \sigma^{xy} : (x,y) \in \supp(\pi_o) \}$. 
\end{proposition}
\begin{remark}[On Proposition \ref{prop: weak conv}]
A version of this proposition was proved by \cite{mikami2004monge} under the extra assumption that $\mu_0$ is absolutely continuous. Theorem 3.7 in \cite{leonard2012schrodinger} implies the proposition but the proof is somewhat involved (as it covers more general settings). We provide a simple proof in Section \ref{sec: proof}.
\end{remark}

We are now in position to state our main results.
Let $H$ denote the space of absolutely continuous maps $h: [0,1] \to \R^d$ with $\int |\dot{h}(t)|^2 dt < \infty$, where $\dot{h}(t) = dh(t)/dt$. We endow $H$ with the (semi-)inner product
\[
( g,h )_H = \int_0^1 \dot{g}(t) \cdot \dot{h}(t) \, dt.
\]
Set $\| \cdot \|_{H} = \sqrt{(\cdot,\cdot )_H}$. Formally, define $\| h \|_{H} = \infty$ for $h \in E \setminus H$. We first state the weak-type LDP for Schr\"{o}dinger bridges, which allows for marginals with unbounded supports. 

\begin{theorem}[Weak-type LDP for Schr\"{o}dinger bridges]
\label{thm: main}
Suppose Assumptions \ref{asp: finite entropy} and \ref{asp: uniqueness} hold. Pick any $\varepsilon_k \downarrow 0$.  Then the following hold. 
\begin{enumerate}
\item[(i)] For every open set $A \subset e_{01}^{-1}(\calX_o \times \calY_o)$ (w.r.t. the relative topology),
\[
\liminf_{k \to \infty} \varepsilon_k \log P^{\varepsilon_k} (A) \ge - \inf_{h \in A} I(h)
\]
for the rate function 
\[
I(h) = 
\frac{\|h\|_H^2}{2} - \psi^c(h(0)) - \psi(h(1)).
\]
\item[(ii)] For every closed set $A \subset E$ of the form $A = e_{01}^{-1}(C)$ for some compact set $C \subset \calX_o \times \calY_o$,
\begin{equation}
\limsup_{k \to \infty} \varepsilon_{k}\log P^{\varepsilon_k} (A) \le -\inf_{h \in A} I(h). 
\label{eq: LDP upper}
\end{equation}
\end{enumerate}
\end{theorem}

Theorem \ref{thm: main} is not precisely a weak LDP since (ii) holds for every compact set $C \subset e_{01}^{-1}(\calX_o \times \calY_o)$ but also for some noncompact closed sets. As such, we call Theorem \ref{thm: main} a weak-type LDP. 
If the marginals have compact supports, then a full LPD holds, subject to one technical condition essential to guarantee uniqueness of OT potentials.

\begin{corollary}[Full LDP for Schr\"{o}dinger bridges]
\label{thm: main2}
Suppose Assumption \ref{asp: finite entropy} holds. Pick any $\varepsilon_k \downarrow 0$. If $\calX$ and $\calY$ are compact and one of them agrees with the closure of a connected open set, then the sequence $\{ P^{\varepsilon_k} \}_{k \in \NN}$ satisfies a (full) LDP on $E$ with speed $\varepsilon_k^{-1}$ and good rate function $I$, where $I$ is set to $\infty$ outside $e_{01}^{-1}(\calX \times \calY)$.
\end{corollary}

We leave several remarks on the preceding results.

\begin{remark}[On Corollary \ref{thm: main2}]
\label{rem: main}
The sets $\calX, \calY$ being compact implies $\calX_o = \calX$ and $\calY_o = \calY$, as the projections from $\calX \times \calY$ onto $\calX$ and $\calY$ are then closed maps. The assumption of Corollary \ref{thm: main2} guarantees uniqueness of OT potentials; see the discussion after Assumption \ref{asp: uniqueness}. Since $P^{\varepsilon}$ charges no mass outside $e_{01}^{-1}(\calX \times \calY)$, the full LDP is indeed deduced from the preceding theorem. 
Connectedness of the support of one of the marginals is essential for uniqueness of OT potentials (see the
discussion before Proposition B.2 in \cite{bernton2021entropic}).
The full LPD (more specifically, establishing exponential tightness) for Schr\"{o}dinger bridges requires both marginals to be compactly supported, since exponential tightness implies the limiting law $P^o$ to be concentrated on a compact set, which fails to hold if one of the marginals has unbounded support. See Remark 4.2 (b) of \cite{nutz2021entropic} for a relevant discussion in the static case.
\end{remark}

\begin{remark}[On the rate function $I(h)$]
Since $\psi^c(x) + \psi(y) \le c(x,y)$ by construction, the rate function $I(h)$ is positive as soon as $h \ne \sigma^{h(0),h(1)}$. Even when $h =\sigma^{h(0),h(1)}$, which entails $\| h \|_H^2/2 = c(h(0),h(1))$, the rate function $I(h) = c(h(0),h(1)) - \psi^c (h(0)) - \psi(h(1))$ can be positive provided $(h(0),h(1)) \notin \supp(\pi_o)$. Section 5 of \cite{bernton2021entropic} provides several conditions under which the rate function for the static case, $\phi (x,y) = c(x,y) - \psi^c(x) - \psi(y)$, is positive outside $\supp(\pi_o)$. Considering the characterization of the support of $P_o$,  our large deviation results essentially imply that the Schr\"{o}dinger bridges $P^{\varepsilon}$ charge exponentially small masses outside $\supp(P_o)$ when $\varepsilon \downarrow 0$. 
\end{remark}

\begin{remark}[Proofs of Theorem \ref{thm: main} and Corollary \ref{thm: main2}]
The proof of Theorem \ref{thm: main} uses the expression $P^{\varepsilon} (A)= \int R^{\varepsilon,xy}(A) \, d\pi_{\varepsilon}(x,y)$ from (\ref{eq: mixture}). The main ingredient is exponential continuity of $\{ R^{\varepsilon_k,xy} \}$, i.e., establishing large deviation upper and lower bounds for $\{ R^{\varepsilon_k,(x_k,y_k)} \}_{k \in \NN}$ when $(x_k,y_k) \to (x,y)$, which will be proved in Proposition \ref{prop: exponential} below. The proof then directly evaluates $P^\varepsilon (A)$ by combining the large deviation results for the static case from \cite{bernton2021entropic}.
As noted in Remark \ref{rem: main}, Corollary \ref{thm: main2} is a special case of Theorem \ref{thm: main}. Nonetheless, we provide a separate, more direct proof for the compact support case. It relies on the expression $P^{\varepsilon}(A) = \int_{A} e^{-\phi_{\varepsilon} \circ e_{01}(\omega)/\varepsilon} \, d\bar{R}^{\varepsilon}(\omega)$ from (\ref{eq: S bridge density}).
Then the proof proceeds 
 as (i) proving an LDP for $\bar{R}^{\varepsilon}$, which follows directly from the exponential continuity \cite{dinwoodie1992large}, and then (ii) adapting the (Laplace-)Varadhan lemma (cf. Theorem 4.4.2 in \cite{dembo2009large}) to evaluate $P^{\varepsilon}(A)$. Step (ii) is relatively simple, because, while the function $\phi_{\varepsilon}$ depends on $\varepsilon$, so the Varadhan lemma is not directly applicable, the assumption of Corollary \ref{thm: main2} ensures uniform convergence of the EOT potentials.  
\end{remark}

{
\begin{remark}[On uniqueness of OT potentials]
Inspection of the proof of Corollary \ref{thm: main2} reveals that, as long as Assumption \ref{asp: finite entropy} holds and $\calX$ and $\calY$ are compact (but without assuming uniqueness of OT potentials), the conclusion of Corollary \ref{thm: main2} continues to hold, provided that 
\begin{equation}
\lim_{k \to \infty} \varphi_{\varepsilon_k} = \bar{\varphi} \quad \text{and} \quad \lim_{k \to \infty} \psi_{\varepsilon_k} = \bar{\psi} \quad \text{uniformly on $\calX$ and $\calY$, respectively},
\label{eq: necessary}
\end{equation}
for some (continuous) functions $\bar{\varphi}$ and $\bar{\psi}$ on $\calX$ and $\calY$, respectively (necessarily, $(\bar{\varphi},\bar{\psi})$ are dual potentials for $(\mu_0,\mu_1)$). The rate function $I$ needs to be modified so that $(\psi^c,\psi)$ are replaced with $(\bar{\varphi},\bar{\psi})$. A similar comment applies to Proposition \ref{prop: langevin} ahead. 
Conversely, the uniform convergence of the EOT potentials in (\ref{eq: necessary}) is necessary for the LDP for the Schr\"{o}dinger bridges $\{ P^{\varepsilon_k} \}_{k \in \NN}$ to hold, by Proposition 4.5 in \cite{nutz2021entropic} in the static case and the contraction principle. 
\end{remark}
}

We shall look at a few special cases.

\begin{example}[F\"{o}llmer process]
When $\mu_0 = \delta_0$, one has $\calY_o = \calY$ and $\psi (y) = |y|^2/2$, so the rate function reduces to $I(h) = \| h \|_H^2/2 - |h(1)|^2/2 + \iota_{\{ 0 \}}(h(0)) + \iota_{\calY}(h(1))$, which vanishes if and only if $h(t)  = ty$ for $t \in [0,1]$ for some $y \in \calY$, i.e., if and only if $h \in \supp(P^o)$.  
\end{example}

\begin{example}[Two-point marginal]
The LDP in Corollary \ref{thm: main2} directly yields an LDP for $P^{\varepsilon_k} \circ f^{-1}$ for any continuous function $f$ from $E$ into another metric space by the contraction principle (cf. Theorem 4.2.1 in \cite{dembo2009large}). We consider the case where $f=e_{st}$ for $0 \le s < t \le 1$. 
Note that $P_{st}^{\varepsilon_k}$ is a coupling for $P_{s}^{\varepsilon_k}$ and $P_{t}^{\varepsilon_k}$. Recall that the marginal flow $(P_t^{\varepsilon})_{t \in [0,1]}$ is called an entropic interpolation, and its limiting analog $(P_t^o)_{t \in [0,1]}$ is a displacement interpolation connecting $\mu_0$ and $\mu_1$. To characterize the rate function for $P_{st}^{\varepsilon_k}$, we need additional notation. 

For a function $f: \R^d \to (-\infty,\infty]$ and $t \ge 0$, define 
\[
\mathcal{Q}_{t}(f)(y) = \inf_{x \in \R^d} \left \{ \frac{c(x,y)}{t} + f(x) \right \}, \ t > 0, \quad \mathcal{Q}_{0}(f)=f.
\] 
The family of operators $\{ \mathcal{Q}_t \}_{t \ge 0}$ is called the \textit{Hopf-Lax semigroup}; cf. Chapter 7 in \cite{villani2009optimal}. Assuming Case (A) after Assumption \ref{asp: uniqueness}, we set $\varphi = \psi^c$ and extend $\varphi$ and $\psi$ to the whole $\R^d$ by setting $\varphi=-\infty$ and $\psi=-\infty$ outside $\calX$ and $\calY$, respectively.  For $0 \le s < t \le 1$, consider the rescaled cost $c^{s,t}(x,t) = c(x,y)/(t-s)$. 

\begin{proposition}[LDP for two-point marginal]
\label{prop: two point}
Suppose Assumption \ref{asp: finite entropy} holds. Pick any $0 \le s < t \le 1$ and $\varepsilon_k \downarrow 0$. If $\calX$ and $\calY$ are compact and one of them agrees with the closure of a connected open set, then the sequence $\{ P^{\varepsilon_k}_{st} \}_{k \in \NN}$ satisfies an LDP on $\R^{2d}$ with speed $\varepsilon_k^{-1}$ and good rate function 
$
I_{st}(x,y) = c^{s,t}(x,y) - \varphi_s(x) - \psi_t(y),
$
where $(\varphi_s,\psi_t):=(-\mathcal{Q}_s (-\varphi),-\mathcal{Q}_{1-t}(-\psi))$ are dual potentials for $(P_s^o,P_t^o)$ w.r.t. $c^{s,t}$, i.e., optimal solutions to (\ref{eq: dual}) with $(\mu_0,\mu_1,c)$ replaced by $(P_s^o,P_t^o,c^{s,t})$.
\end{proposition}
\end{example}

{
Finally, we shall point out that the direct proof for Corollary \ref{thm: main2} can be easily adapted to cover the dynamical Schr\"{o}dinger problem with Langevin diffusion as a reference measure. 
}

\begin{remark}[Langevin diffusion as reference measure]
\label{rem: langevin}
{
For a bounded smooth potential $V: \R^d \to \R$ with bounded derivatives, consider the Langevin diffusion $X=(X(t))_{t \ge 0}$ defined by the unique (strong) solution to the following SDE:
\[
dX(t) = -\nabla V(X(t)) \, dt + \, dW(t), \ X(0) \sim \mu_0,
\]
where $(W(t))_{t \ge 0}$ is a standard Brownian motion starting at $0$ independent of $X(0)$. Let $p_{t}(x,y)$ denote the transition density of the Langevin diffusion $X$ and $\check{R}^\varepsilon$ be the law of $X^{\varepsilon} := (X(\varepsilon t))_{t \in [0,1]}$ defined on $\calB(E)$. Instead of the Wiener reference measure as in (\ref{eq: dynamical Schrodinger}), we consider the dynamical Schr\"{o}dinger problem with reference measure $\check{R}^{\varepsilon}$:
\begin{equation}
\min_{P: P_0 = \mu_0,P_1=\mu_1} \calH(P|\check{R}^{\varepsilon}). 
\label{eq: langevin}
\end{equation}
Under Assumption \ref{asp: finite entropy}, arguing as in Section \ref{sec: prelim}, one can see that the unique optimal solution to (\ref{eq: langevin}) is given by
\[
\check{P}^{\varepsilon} (\cdot) = \int \check{R}^{\varepsilon,xy}(\cdot) \, d\check{\pi}^{\varepsilon}(x,y),
\]
where $\check{R}^{\varepsilon,xy}$ is the conditional law of $X^\varepsilon$ given $(X^\varepsilon (0),X^{\varepsilon}(1))=(x,y)$ and $\check{\pi}^\varepsilon$ is the unique optimal solution to the static EOT problem
\[
\min_{\pi \in \Pi(\mu_0,\mu_1)}\int c_{\varepsilon} \, d\pi + \varepsilon \calH(\pi | \mu_0\otimes \mu_1)
\]
with $c_{\varepsilon}(x,y):=-\varepsilon \log p_\varepsilon (x,y)$. Recall that the transition density $p_t (x,y)$ is everywhere positive (cf. Chapter 3 in \cite{stroock2008partial}) and the conditional laws $\check{R}^{\varepsilon,xy}$ are defined for all $(x,y) \in \R^{2d}$ \cite{bravo2011markovian}. The classical Varadhan asymptotics implies that $\lim_{\varepsilon \downarrow 0} c_\varepsilon (x,y) = |x-y|^2/2=c(x,y)$ (cf. Chapter 4 in \cite{stroock2008partial}), so one can expect that the Schr\"{o}dinger bridges $\{ \check{P}^\varepsilon \}_{\varepsilon > 0}$ satisfy the LDP with the same rate function $I$ as in the Brownian case.  The next proposition confirms this under a similar setting to Corollary \ref{thm: main2}.
}
{
\begin{proposition}[Full LDP for Schr\"{o}dinger bridges: Langevin case]
\label{prop: langevin}
Suppose Assumption \ref{asp: finite entropy} holds. Pick any $\varepsilon_k \downarrow 0$. If $\calX$ and $\calY$ are compact and one of them agrees with the closure of a connected open set, then the sequence $\{ \check{P}^{\varepsilon_k} \}_{k \in \NN}$ satisfies a (full) LDP on $E$ with speed $\varepsilon_k^{-1}$ and good rate function $I$, where $I$ is given in Corollary \ref{thm: main2}.
\end{proposition}
}
{
The condition on the potential $V$ appears to be stronger than needed, but is imposed for the sake of simplicity. 
As announced, the proof follows similar arguments to the direct proof for Corollary \ref{thm: main2}. To establish exponential continuity for the Langevin bridge $\check{R}^{\varepsilon,xy}$, we use the explicit expression of the Radon-Nikodym density of the Langevin bridge against the Brownian bridge; cf. \cite{levy1993dynamics}.}
\end{remark}

\section{Proofs for Section \ref{sec: main}}
\label{sec: proof}
Recall that $R^{\varepsilon,xy}$ is the (regular) conditional law of $x+\sqrt{\varepsilon}W$ given $x+\sqrt{\varepsilon}W(1) = y$ for a standard Brownian motion $W=(W(t))_{t \in [0,1]}$ starting at $0$.  Alternatively, $R^{\varepsilon,xy}$ can be characterized as the law of $\sqrt{\varepsilon} W^\circ + \sigma^{xy}$ with $W^\circ = (W(t)-tW(1))_{t \in [0,1]}$, a standard Brownian bridge. 
For simplicity of notation, let $z = (x,y) \in \R^{2d}$ and denote $R^{\varepsilon,z} = R^{\varepsilon,xy}$. 

\subsubsection{Proof of Proposition \ref{prop: weak conv}}

By uniqueness of the OT plan, we have $\pi_{\varepsilon} \to \pi_o$ weakly by Proposition 3.2 in \cite{bernton2021entropic}, which implies that
\[
\eta_{\varepsilon} := \sup_{\substack{g: \R^{2d} \to [-1,1] \\ g \ \text{$1$-Lipschitz}} } \left | \int g \, d(\pi_{\varepsilon} - \pi_o) \right | \to 0.
\]
See Chapter 1.12 in \cite{vandervaart1996weak}.
Pick any $1$-Lipschitz function $f: E \to [-1,1]$. 
One has
\[
\int f \, dP^{\varepsilon} = \int \left(\int f \, dR^{\varepsilon,z}\right) \, d\pi_{\varepsilon}(z) =\int \underbrace{\E \left [ f( \sqrt{\varepsilon} W^{\circ} + \sigma^z) \right]}_{=:g_{\varepsilon}(z)} \, d\pi_{\varepsilon}(z).
\]
By construction, $g_\varepsilon$ is bounded by $1$,  $|g_\varepsilon (z) - g_\varepsilon (z')| \le \| \sigma^z - \sigma^{z'} \|_{E} \le 2| z-z'|$, and $\lim_{\varepsilon \downarrow 0} g_{\varepsilon}(z)= f(\sigma^z) = \int f \, d\delta_{\sigma^z}$. Hence,
\[
\int g_{\varepsilon} \, d\pi_{\varepsilon} \le \int g_{\varepsilon} \, d\pi_o + 2\eta_\varepsilon = \underbrace{\int \left (\int f \, d\delta_{\sigma^z} \right ) \, d\pi_o}_{=\int f \, dP^o} + o(1)
\]
where we used the dominated convergence theorem. The reverse inequality follows similarly, and we conclude that 
\[
\lim_{\varepsilon \downarrow 0} \int f \, dP^{\varepsilon} = \int f \, dP^o,
\]
which yields $P^\varepsilon \to P^o$ weakly. The second claim follows from Lemma \ref{lem: support} below. \qed

\begin{lemma}
\label{lem: support}
For any Borel probability measure $\gamma$ on $\R^{2d}$, the mixture $P = \int \delta_{\sigma^{xy}} \, d\gamma(x,y)$ has support $\Sigma_{\gamma} := \{ \sigma^{xy} : (x,y) \in \supp (\gamma) \}$. 
\end{lemma}

\begin{proof}
The set $\Sigma_{\gamma}$ is closed in $E$. Pick any $(x,y) \in \supp(\gamma)$ and any open set $U$ containing $\sigma^{xy}$. Since $O = \{ (x',y') : \sigma^{x'y'} \in U \}$ is open in $\R^{2d}$ (as $(x',y') \mapsto \sigma^{x'y'}$ is continuous), we have, for $(\xi_0,\xi_1) \sim \gamma$, $P(U) = \Prob(\sigma^{\xi_0,\xi_1} \in U) = \Prob((\xi_0,\xi_1) \in O) = \gamma(O) > 0$, which yields $\supp(P) = \Sigma_{\gamma}$. 
\end{proof}

\subsection{Exponential continuity of Brownian bridges}

For given $x,y \in \R^d$, \cite{hsu1990brownian} showed that the sequence $\{ R^{\varepsilon,xy} \}_{\varepsilon > 0}$ satisfies an LDP with rate function
\[
J_{xy}(h) =
\frac{\| h \|_{H}^2}{2} - c(x,y) + \iota_{\{(x,y)\}}(h(0),h(1)).
\]
Denote $J_z(h) = J_{xy}(h)$ for $z= (x,y)$. Additionally, set $H_z := \{ h \in H : (h(0),h(1)) = z \}$. Pick any $\varepsilon_k \downarrow 0$. 

\begin{proposition}[Exponential continuity of Brownian bridges]
\label{prop: exponential}
The following hold.
\begin{enumerate}
    \item[(i)] For every open set $A \subset E$, 
\[
\liminf_{k \to \infty} \varepsilon_k \log R^{\varepsilon_k,z_k} (A) \ge -\inf_{h \in A} J_{z}(h)
\]
whenever $z_k \to z$ in $\R^{2d}$.
\item[(ii)] For every closed set $A \subset E$, 
\begin{equation}
\limsup_{k \to \infty} \varepsilon_k \log R^{\varepsilon_k,z_k} (A) \le -\inf_{h \in A} J_{z}(h)
\label{eq: upper bound zero}
\end{equation}
whenever $z_k \to z$ in $\R^{2d}$.
\end{enumerate}
\end{proposition}

\begin{proof}
Hsu's proof in \cite{hsu1990brownian} that relies on transition function estimates seems difficult to adapt to establishing the exponential continuity. Instead, we adapt the proof of large deviations for abstract Wiener spaces; cf.  Chapter 8 in \cite{stroock2010probability}. {For the sake of completeness, we provide a self-contained proof.}  

(i). It suffices to show that for every $h \in H$ such that $J_z (h) < \infty$, 
\[
\liminf_{r \downarrow 0} \liminf_{k \to \infty} \varepsilon_k \log R^{\varepsilon_k,z_k}(B_E(h,r)) \ge -J_z (h). 
\]
Set $\bar{h} \in H_0$ by $\bar{h}= h -\sigma^{xy}$ and $h_k \in H_{z_k}$ by $h_k = \bar{h} + \sigma^{x_k,y_k}$. Since $\| h_k - h \|_E \to 0$, $B_E(h_k,r/2) \subset B_E(h,r)$ for large $k$. Observe that
\[
R^{\varepsilon_k,z_k}\big(B_E(h_k,r/2)\big) = R^{\varepsilon_k,0}\big(B_E(\bar{h},r/2)\big)
= \Prob \Big(W^\circ \in B_E\big(\bar{h}/\sqrt{\varepsilon_k},r/(2\sqrt{\varepsilon_k})\big)\Big).
\]
Recall that $(H_0, (\cdot,\cdot)_H)$ is a reproducing kernel Hilbert space for $W^\circ$ (cf. Exercise 2.6.16 in \cite{gine2021mathematical}), whose closure in $E$ agrees with $E_0 := \{ \omega \in E : \omega (0)=\omega(1)=0 \}$. Hence, the pair of spaces $(H_0,E_0)$ coupled with the law of $W^\circ$ constitutes an abstract Wiener space; cf. Chapter 8 in \cite{stroock2010probability}. {Let $E_0^*$ denote the topological dual of $E_0$ with dual norm $\| \cdot \|_{E_0^*}$ and $\langle \omega,\omega^* \rangle$ denote the duality pairing for $\omega \in E_0$ and $\omega^* \in E_0^*$. Since $H_0$ is continuously embedded as a dense subspace of $E_0$ (as $\| \cdot \|_E \le \| \cdot \|_H$ on $H_0$), for each $\omega^* \in E_0^*$, there exists a unique $h_{\omega^*} \in H_0$ with the property that $(h,h_{\omega^*})_H = \langle h,\omega^* \rangle$ for all $h \in H_0$, and the map $\omega^* \mapsto h_{\omega^*}$ is continuous, linear, one-to-one, and onto a dense subspace of $H_0$ (cf. Lemma 8.2.3 in \cite{stroock2010probability}). Let $0 < \delta < r/2$ and  $\omega^* \in E_0^*$ be such that $B_{E_0}(h_{\omega^*},\delta) \subset B_{E_0}(\bar{h},r/2)$. 
Now, an application of the Cameron-Martin formula (cf. Theorem 8.2.9  in \cite{stroock2010probability}) yields 
\[
\begin{split}
R^{\varepsilon_k,0}\big(B_E(\bar{h},r/2)\big) &= R^{\varepsilon_k,0}\big(B_{E_0}(\bar{h},r/2)\big)
\ge R^{\varepsilon_k,0}\big(B_{E_0}(h_{\omega^*},\delta)\big) \\
&=  \Prob \Big(W^\circ-\varepsilon_k^{-1/2}h_{\omega^*} \in B_{E_0}\big(0,\varepsilon_k^{-1/2}\delta)\big)\Big) \\
&= \E\left[ e^{-\varepsilon_k^{-1/2}\langle W^\circ,\omega^*\rangle -\varepsilon_k^{-1}\|h_{\omega^*}\|_{H}^2/2}\bm{1}_{B_{E_0}(0,\varepsilon_k^{-1/2}\delta)}(W^\circ) \right ] \\
&\ge e^{-\delta \varepsilon_k^{-1}\| \omega^* \|_{E_0^*}-\varepsilon_k^{-1}\|h_{\omega^*}\|_{H}^2/2} \Prob \Big (W^\circ \in B_{E_0}(0,\varepsilon_k^{-1/2}\delta) \Big),
\end{split}
\]
so that, by taking $k \to \infty$, 
\[
\liminf_{k \to \infty} \varepsilon_k \log R^{\varepsilon_k,0}\big(B_E(\bar{h},r/2)\big) \ge - \delta \|\omega^*\|_{E_0^*} - \frac{\|h_{\omega^*}\|_H^2}{2}. 
\]
Choosing $\delta = r/4$ and $\omega^* \in E_0^*$ with $\| \bar{h}-h_{\omega^*} \|_{H} < r/4$, and then taking $r \downarrow 0$, one has 
\[
\begin{split}
&\liminf_{r \downarrow 0} \liminf_{k \to \infty} \varepsilon_k \log R^{\varepsilon_k,0}\big(B_E(\bar{h},r/2)\big) \\
&\ge - \frac{\| \bar{h} \|_H^2}{2} = -\frac{1}{2} \left ( \| h \|_{H}^2 - |x-y|^2 \right) = - J_z(h). 
\end{split}
\]
}

(ii). We first show that for every $h \in E$,
\begin{equation}
\limsup_{r \downarrow 0} \limsup_{k \to \infty} \varepsilon_k \log R^{\varepsilon_k,z_k} (B_E(h,r)) 
\le -J_z(h). 
\label{eq: upper bound}
\end{equation}
{Using the same notation as in (i), we have $B_E(h,r) \subset B_E(h_k,2r)$ for large $k$ and 
\[
\begin{split}
R^{\varepsilon_k,z_k}\big(B_E(h_k,2r)\big) &= \Prob \Big(W^\circ \in B_E\big(\bar{h}/\sqrt{\varepsilon_k},2r/\sqrt{\varepsilon_k})\big)\Big) \\
&= \E \left [ e^{-\varepsilon_k^{-1/2}\langle W^\circ,\omega^*\rangle +\varepsilon_k^{-1/2}\langle W^\circ,\omega^*\rangle} \bm{1}_{B_{E_0}(\bar{h}/\sqrt{\varepsilon_k},2r/\sqrt{\varepsilon_k})}(W^\circ) \right ] \\
&\le e^{-\varepsilon_k^{-1}(\langle \bar{h},\omega^* \rangle -2r\|\omega^*\|_{E_0^*})}\E\left [ e^{\varepsilon_k^{-1/2}\langle W^\circ,\omega^*\rangle}\right ] \\
&= e^{-\varepsilon_k^{-1}(\langle \bar{h},\omega^* \rangle -\| h_{\omega^*}\|_H^2/2 - 2r\|\omega^*\|_{E_0^*})}
\end{split}
\]
for all $\omega^* \in E_0^*$, where we used the fact that $\langle W^\circ,\omega^*\rangle \sim N(0,\|h_{\omega^*}\|_{H}^2)$. 
This yields
\[
\begin{split}
\limsup_{r \downarrow 0} \limsup_{k \to \infty} \varepsilon_k \log R^{\varepsilon_k,z_k}\big(B_E(h_k,2r)\big) 
&\le -\sup_{\omega^* \in E_0^*} \left ( \langle \bar{h},\omega^* \rangle -\frac{\| h_{\omega^*}\|_H^2}{2} \right ) \\
&=
\begin{cases}
-\frac{\|\bar{h}\|_H^2}{2} & \text{if $\bar{h} \in H_0$}, \\
-\infty & \text{otherwise}.
\end{cases}
\end{split}
\]
}
Now, $\bar{h} \in H_0$ if and only if $h \in H_z$, and $\| \bar{h} \|_H^2=\| h \|_H^2 - |x-y|^2$, which leads to (\ref{eq: upper bound}). 

Given (\ref{eq: upper bound}), it is standard to show that (\ref{eq: upper bound zero}) holds for every compact set $A \subset E$. It remains to verify exponential tightness for $\{ R^{\varepsilon_k,z_k} \}_{k \in \NN}$ (cf. Lemma 1.2.18 in \cite{dembo2009large}), i.e., for every $\alpha < \infty$, there exists a compact set $K \subset E$ such that $\limsup_{k \to \infty} \varepsilon_k \log R^{\varepsilon_k,z_k} (K^c) < -\alpha$.
    {We first note that the exponential tightness holds for $\{ R^{\varepsilon_k,0} \}_{k \in \NN}$. Indeed, by Corollary 8.3.10 in \cite{stroock2010probability}, one can construct a separable Banach space $F$ that is continuously embedded in $E_0$ as a measurable subset with the properties that  $\Prob (W^\circ \in F)=1$, bounded subsets of $F$ are totally bounded in $E_0$, and $(H_0,F)$ coupled with the restriction of the law of $W^\circ$ on $F$ is another abstract Wiener space. Then, choosing $K_0$ to be the $E_0$-closure of a  ball in $F$ with large enough radius satisfies $\limsup_{k \to \infty} \varepsilon_k \log R^{\varepsilon_k,0} (K_0^c) < -\alpha$ by Fernique's theorem (cf. Theorem 8.2.1 in \cite{stroock2010probability}), and $K_0$ is compact in $E_0$ by construction.}  

Now, for an arbitrary bounded neighborhood $O \subset \R^{2d}$ of $z$, set $K_1=\{ \sigma^{x'y'} : (x',y') \in O \}$.
By the Ascoli-Arzel\`{a} theorem, the set $K=\{ \omega + \omega' : \omega \in K_0, \omega' \in K_1 \}$
is relatively compact in $E$, and such that $R^{\varepsilon_k,z_k}(K) {\ge} R^{\varepsilon_k,0} (K_0)$ for large $k$.
{Indeed, $\sigma^{z_k} \in K_1$  for large $k$, so if $\omega \in K_0$, then $\omega + \sigma^{z_k} \in K$, which implies $R^{\varepsilon_k,0}(K_0) \le R^{\varepsilon_k,0} (\omega+\sigma^{z_k} \in K) = R^{\varepsilon_k,z_k}(K)$.} 
This yields exponential tightness for $\{ R^{\varepsilon_k,z_k} \}_{k \in \NN}$. 
\end{proof}

Given the exponential continuity, the following corollary concerning large deviations of mixtures of Brownian bridges follows immediately from Theorems 2.1 and 2.2 in \cite{dinwoodie1992large}.
The result might be of independent interest.

\begin{corollary}[Large deviations for mixtures of Brownian bridges]
\label{prop: LDP B bridge}
Let $\gamma$ be a Borel probability measure on $\R^{2d}$. Consider the mixture distribution
$Q^{\varepsilon} = \int R^{\varepsilon,xy} \, d\gamma(x,y)$. 
Then the following hold.

\begin{enumerate}
    \item[(i)] The function 
\begin{equation}
J(h)  := \inf_{(x,y) \in \supp (\gamma)} J_{xy}(h)  \\
=
\frac{\|h\|_H^2}{2} - c(h(0),h(1)) + \iota_{\supp(\gamma)}(h(0),h(1))
\label{eq: rate function}
\end{equation}
is lower semicontinuous from $E$ into $[0,\infty]$.
\item[(ii)] For every open set $A \subset E$, 
\[
\liminf_{k \to \infty} \varepsilon_k\log Q^{\varepsilon_k} (A) \ge -\inf_{h \in A} J(h).
\]
\item[(iii)] If $\gamma$ is compactly supported, then for every closed set $A \subset E$,
\[
\limsup_{k \to \infty} \varepsilon_k\log Q^{\varepsilon_k} (A) \le -\inf_{h \in A} J(h),
\]
and $J$ is a good rate function.
\end{enumerate}
\end{corollary}

\begin{proof}
(i).  Set $F = \{ (h,z) \in E \times \R^{2d} : (h(0),h(1)) = z \}$. The rate function $J_{z}(h)$ can be expressed as 
\begin{equation}
J_{z}(h) = \frac{\|h\|_H^2}{2} - c(x,y) + \iota_{F}(h,z) = \frac{\|h\|_H^2}{2} - c(h(0),h(1)) + \iota_{F}(h,z).
\label{eq: J function}
\end{equation}
This yields the second expression of the $J$ function in (\ref{eq: rate function}). Since $\supp(\gamma)$ is closed by definition, what remains is to verify that the mapping $E \ni h \mapsto \| h \|_{H}^2/2$ is lower semicontinuous. It suffices to show that the set $\{ h \in H : \| h \|_H \le 1 \}$ is closed  in $E$. Let $\{ h_n \}_{n \in \NN} \subset H$ be a sequence with $ \| h_n \|_H \le 1$ for all $n \in \NN$ and $h_n \to h_\infty$ in $E$. We may assume without loss of generality that $h_n(0)=h_\infty(0)=0$. Since $\tilde{H}= \{ h \in H : h(0) = 0\}$ endowed with inner product $( \cdot, \cdot )_H$ is a Hilbert space, by the Banach-Alaoglu theorem, there exists a subsequence $h_{n'}$ such that $h_{n'} \to \tilde{h}$ weakly in $\tilde{H}$ for some $\tilde{h} \in \tilde{H}$ with $\| \tilde{h} \|_H \le 1$, i.e., $\lim_{n'} (h_{n'},g)_H = (\tilde{h},g)_H$ for all $g \in \tilde{H}$. This implies $h_\infty = \tilde{h}$ (choose appropriate $g$) and $\| h_\infty \|_{H} \le 1$, as desired.  

(ii). This follows from Proposition \ref{prop: exponential} (i) above and Theorem 2.1 in \cite{dinwoodie1992large}.

(iii). The large deviation upper bound follows from Proposition \ref{prop: exponential} (ii) above and Theorem 2.2 in \cite{dinwoodie1992large}.
 Finally, we shall verify that $J$ has compact level sets, but this follows from Lemma 1.2.18 in \cite{dembo2009large}, since
 the argument in Proposition \ref{prop: exponential} (ii) indeed shows that $\{ Q^{\varepsilon_k} \}_{k \in \NN}$ is exponentially tight (replace $O$ by $\supp(\gamma)$).

\end{proof}

\subsection{Proof of Theorem \ref{thm: main}}
Set $\phi(z) = c(x,y) - \psi^c(x) - \psi(y)$ for $z = (x,y) \in \calX_o \times \calY_o$.

(i). It suffices to show that for any $h \in e_{01}^{-1}(\calX_o \times \calY_o) \cap H$ and $r > 0$, 
\[
\liminf_{k \to \infty} \varepsilon_k \log P^{\varepsilon_k}(B_E(h,r)) \ge - I(h). 
\]
Set $z = (h(0),h(1)) \in \calX_o \times \calY_o$.
By exponential continuity of $\{ R^{\varepsilon_k,z} \}$ established in Proposition \ref{prop: exponential}, for every $\delta >0$, one can choose an open neighborhood $O_z \subset \calX_o \times \calY_o$ of $z$ and a positive integer $k_z$ such that for every $z' \in O_z$,
\[
\varepsilon_k \log R^{\varepsilon_k,z'}(B_E(h,r)) \ge -\inf_{h' \in B_E(h,r)} J_z (h')- \delta, \quad k \ge k_z.
\]
{
For if not, for the open ball $O_i$ in $\calX_\circ \times \calY_\circ$ with center $z$ and radius $i^{-1}$, one can find $z_i' \in O_i$ and a large enough positive integer $k_i$ (with $k_i > k_{i-1}$) such that 
\[
\varepsilon_{k_i} \log R^{\varepsilon_{k_i},z_i'}(B_E(h,r)) < -\inf_{h' \in B_E(h,r)} J_z (h')- \delta,
\]
but this contradicts the exponential continuity (as $z_i' \to z$).} 
Hence,
\[
\begin{split}
P^{\varepsilon_k} (B_E(h,r)) &\ge \int_{O_z} e^{\varepsilon_k^{-1} \cdot \varepsilon_k \log R^{\varepsilon_k,z'}(B_E(h,r)) } d\pi_{\varepsilon_k}(z') \\
&\ge e^{-\varepsilon_k^{-1}(\inf_{h' \in B_E(h,r)} J_z (h')+ \delta)} \pi_{\varepsilon_k}(O_z).
\end{split}
\]
Invoking Corollary 4.7 in \cite{bernton2021entropic}, we arrive at
\[
\begin{split}
\liminf_{k \to \infty} \varepsilon_k \log P^{\varepsilon_k} (B_E(h,r)) &\ge - \inf_{h' \in B_E(h,r)} J_z (h')- \delta - \inf_{z' \in O_z}\phi(z') \\
&\ge -(J_z(h)+\phi(z)) - \delta \\
&= -I(h) - \delta,
\end{split}
\]
establishing the desired claim. 

(ii). We first observe that for $A = e_{01}^{-1}(C)$ with  $C \subset \calX_o \times \calY_o$ compact, $P^{\varepsilon}(A) = \int_{C} R^{\varepsilon,z}(A) \, d\pi_{\varepsilon}(z)$. 
Taking into account Proposition 4.5 in \cite{bernton2021entropic}, extend $\phi$ to $\calX \times \calY$ as
\[
\phi(x,y) = \sup_{\ell \ge 2} \sup_{\{ (x_i,y_i) \}_{i=1}^\ell \subset \supp(\pi_o)} \sup_{\tau} \sum_{i=1}^\ell c(x_i,y_i) - \sum_{i=1}^\ell c(x_i,y_{\tau(i)}),
\]
where $\sup_{\tau}$ is taken over all permutations of $\{ 1,\dots,\ell \}$ and $(x_1,y_1)=(x,y)$. The function $\phi: \calX \times \calY \to [0,\infty]$ is lower semicontinuous (Lemma 4.2 in \cite{bernton2021entropic}) and agrees with the previous definition of $\phi$ on $\calX_o \times \calY_o$.  
Let $\delta > 0$ be given. For every $z \in C$, 
by exponential continuity of $\{ R^{\varepsilon_k,z} \}$ established in Proposition \ref{prop: exponential}, one can choose a bounded open neighborhood $O_z \subset \calX \times \calY$ of $z$ and a positive integer $k_z$ such that for every $z' \in O_z$,
\[
\varepsilon_k \log R^{\varepsilon_k,z'}(A) \le -\inf_{h \in A} J_z (h) + \delta, \quad k \ge k_z.
\]
Furthermore, since $\phi$ is lower semicontinuous, by choosing $O_z$ smaller if necessary, one has
\[
\inf_{z' \in \bar{O}_z} \phi(z') \ge \phi(z) - \delta,
\]
where $\bar{O}_z$ denotes the closure of $O_z$ in $\calX \times \calY$. 
By compactness of $C$, one can find $z_1,\dots,z_N \in C$ such that $C \subset \bigcup_{i=1}^N O_{z_i}$, so 
\[
P^{\varepsilon_k}(A) \le \sum_{i=1}^N \int_{O_{z_i}} e^{\varepsilon_k^{-1} \cdot\varepsilon_k \log R^{\varepsilon_k,z}(A)}\, d\pi_{\varepsilon_k}(z) \le \sum_{i=1}^N e^{\varepsilon_k^{-1} (-\inf_{h \in A} J_{z_i} (h) + \delta + \varepsilon_k \log \pi_{\varepsilon_k}(\bar{O}_{z_i}))}.
\]
We invoke the following elementary result, whose proof follows from  Jensen's inequality; cf. \cite{chatterjee2005error}. 
\begin{lemma}[Smooth max function]
\label{lem: smooth max}
For $\beta > 0$ and $v = (v_1,\dots,v_N) \in \R^N$, consider a smooth max function $m_\beta(v) = \beta^{-1}\log (\sum_{i=1}^N e^{\beta v_i})$. Then, for every $v \in \R^N$, we have $\max_{1 \le i \le N}v_i \le m_\beta (v) \le \max_{1 \le i \le N}v_i + \beta^{-1}\log N$. 
\end{lemma}
Using the preceding lemma, and combining Corollary 4.3 in \cite{bernton2021entropic}, we have
\[
\begin{split}
\varepsilon_k\log P^{\varepsilon_k}(A) &\le \max_{1 \le i \le N} \big \{  -\inf_{h \in A} J_{z_i} (h) + \delta + \varepsilon_k \log \pi_{\varepsilon_k}(\bar{O}_{z_i}) \big \} + \varepsilon_k \log N \\
&\le \max_{1 \le i \le N} \big \{  -\inf_{h \in A} J_{z_i} (h)  - \inf_{z \in \bar{O}_{z_i}}\phi(z) \big \} + \delta +o(1) \\
&\le \max_{1 \le i \le N} \big \{  -\inf_{h \in A} J_{z_i} (h) - \phi (z_i) \big \} + 2\delta + o(1) \\
&\le -\inf_{h \in A} \inf_{z \in C} \{ J_z(h)+\phi(z)\} + 2\delta + o(1)\\
&=-\inf_{h \in A} \inf_{z \in C} \{ I(h) + \iota_{\{ z \}}(h(0),h(1)) \} + 2\delta + o(1) \\
&= -\inf_{h \in A} I(h) + 2 \delta +o(1),
\end{split}
\]
where we used the fact that $(h(0),h(1)) \in C$ whenever $h \in A$ by our choice of $A$. This completes the proof. 
\qed

\subsection{Direct proof of Corollary \ref{thm: main2}}
We first prove the following technical result concerning convergence of EOT potentials. 
\begin{lemma}[Convergence of EOT potentials]
\label{lem: uniqueness}
Suppose that $\calX$ and $\calY$ are compact and one of them agrees with the closure of a connected open set. Then, under normalization $\int \psi^c \, d\mu_0 = \int \psi \, d\mu_1$, the OT potential $\psi$ from $\mu_1$ to $\mu_0$ is everywhere unique, and $(\psi^c,\psi)$ are bounded and Lipschitz on $\calX \times \calY$.  Furthermore, let $(\varphi_{\varepsilon},\psi_{\varepsilon})$ be the unique EOT potentials under normalization $\int \varphi_{\varepsilon} \, d\mu_0 = \int \psi_{\varepsilon} \, d\mu_1$, then for any sequence $\varepsilon_{k} \downarrow 0$, one has $\varphi_{\varepsilon_k} \to \psi^c$ and $\psi_{\varepsilon_{k}} \to \psi$ uniformly on $\calX$ and $\calY$, respectively. 
\end{lemma}

\begin{proof}
The lemma follows from Proposition 7.18 in \cite{santambrogio15} and Proposition 3.2 in \cite{nutz2021entropic}. We include a self-contained proof for completeness. First, under the current assumption,  we observe that any OT potential $\psi$ is bounded and Lipschitz on $\calY$. We have seen that the support of any OT plan $\pi$ is contained in $\partial^c \psi$, so any  $(x_0,y_0) \in \supp(\pi)$ satisfies $\psi(y_0) > -\infty$ and $\psi^c(x_0) > -\infty$, which entails  $\psi = \psi^{cc} \le \sup_{\calX \times \calY} c -\psi^c(x_0)$ and $\psi \ge - \sup_{\calX} \psi^c \ge -\sup_{\calX \times \calY} c + \psi(y_0)$. Lipschitz continuity follows from $c$-concavity. For the uniqueness, suppose $\mathrm{int}(\calY)$ is connected. Recall that the projections of $\supp(\pi)$ onto $\calX$ and $\calY$ agree with $\calX$ and $\calY$, respectively (cf. Remark \ref{rem: main}).  For any OT potential $\psi$ and  any $y_0 \in \mathrm{int} (\calY)$, one can find $x_0 \in \calX$ such that $\psi^c(x_0) + \psi(y_0) =c(x_0,y_0)$, i.e., $c(x_0,\cdot) - \psi(\cdot)$ is minimized at $y_0$, which entails $\nabla \psi(y_0) = \nabla_y c(x_0,y_0)$ as long as $\psi$ is differentiable at $y_0$. We have shown that $\nabla \psi$ is uniquely determined Lebesgue a.e. on $\mathrm{int} (\calY)$. As $\mathrm{int}(\calY)$ is connected, $\psi$ is uniquely determined on $\mathrm{int}(\calY)$ up to additive constants. By continuity, $\psi$ is uniquely determined on $\calY$ up to additive constants. If $\mathrm{int}(\calX)$ is connected, then the OT potential $\varphi$ from $\mu_0$ to $\mu_1$ is unique up to additive constants. If $\psi$ is an OT potential from $\mu_1$ to $\mu_0$, then by the definition of the $c$-transform, one must have $\int (\psi - \varphi^c) \, d\mu_1 = 0$, which yields $\psi = \varphi^c$ $\mu_1$-a.e. By continuity, we have $\psi = \varphi^c$ on $\calY$. 

 For the latter result, by the Schr\"{o}dinger system (\ref{eq: S system}) and Jensen's inequality, one has $\psi_{\varepsilon}^c \le \varphi_{\varepsilon} \le \sup_{\calX \times \calY} c$ and $\varphi_{\varepsilon}^c \le \psi_{\varepsilon} \le \sup_{\calX \times \calY} c$, so the EOT potentials are uniformly bounded by $\sup_{\calX \times \calY} c$.
 Furthermore, under our assumption, the EOT potentials extend to smooth functions on $\R^d$ by the Schr\"{o}dinger system, and directly calculating derivatives shows that $| \nabla \varphi_{\varepsilon} | \vee | \nabla \psi_{\varepsilon}| \le C$ on $\calX \times \calY$ for some constant $C$ independent of $\varepsilon$. 
 Hence, the Ascoli-Arzel\`{a} theorem applies, and after passing to a subsequence, $\varphi_{\varepsilon_k} \to \bar{\varphi}$ and $\psi_{\varepsilon_k} \to \bar{\psi}$ uniformly on $\calX$ and $\calY$, respectively. By the identity $\int e^{(\varphi_{\varepsilon}+\psi_{\varepsilon}-c)/\varepsilon} \, d(\mu_0 \otimes \mu_1) = 1$ and Fatou's lemma, one has $\bar{\varphi} + \bar{\psi} \le c$ $(\mu_0 \otimes \mu_1)$-a.e. By continuity, $\bar{\varphi} + \bar{\psi} \le c$ on $\calX \times \calY$, but $\bar{\psi}^c \le \bar{\varphi}$ and $\bar{\varphi}^c \le \bar{\psi}$ by construction, so $\bar{\varphi} = \bar{\psi}^c$ and $\bar{\psi} = \bar{\varphi}^c$, i.e., $(\bar{\varphi},\bar{\psi})$ are $c$-concave.
Now, using duality, for any OT plan $\pi$,  $ \int c \, d\pi \le \lim_{k} \big( \int c\, d\pi_{\varepsilon_k} + \varepsilon_k H(\pi_{\varepsilon_k} | \mu_0 \otimes \mu_1) \big)= \int \bar{\varphi} \, d\mu_0 + \int \bar{\psi} \, d\mu_1 \le \int c \, d\pi$, so $(\bar{\varphi},\bar{\psi})$ are OT potentials. Since $\int \bar{\varphi} \, d\mu_0 = \int \bar{\psi} \, d\mu_1$ by construction, by the uniqueness result, $\bar{\varphi} = \psi^c$ and $\bar{\psi}=\psi$.
Finally, by uniqueness of the limits, along the original sequence, $\varphi_{\varepsilon_k} \to \psi^c$ and $\psi_{\varepsilon_k} \to \psi$ uniformly on $\calX$ and $\calY$, respectively.
\end{proof}

\begin{proof}[Direct proof of Corollary \ref{thm: main2}]

Set $S = e_{01}^{-1} (\calX \times \calY) \subset E$. Recall $\bar{R}^{\varepsilon} = \int R^{\varepsilon,xy} \, d(\mu_0 \otimes \mu_1)$. By construction, $\bar{R}^{\varepsilon}(S)=1$ for all $\varepsilon > 0$. Abusing notation, we shall write $\phi_{\varepsilon}(\omega) = \phi_{\varepsilon} (\omega(0),\omega(1))$. With this convention, we have  $P^\varepsilon (A) = \int_A e^{-\phi_\varepsilon/\varepsilon}  \, d\bar{R}^{\varepsilon}$. Set $J(h) = \inf_{(x,y) \in \calX \times \calY} J_{xy}(h)$ and $\phi(h) = \phi(h(0),h(1)) = c(h(0),h(1))-\psi^c (h(0))-\psi(h(1))$ for $h \in S$.

\underline{Step 1}.
Let $A \subset E$ be open and pick any $h \in A$ such that $I(h)<\infty$ (if no such $h$ exists then the conclusion is trivial). By  Lemma \ref{lem: uniqueness}, for every $\delta > 0$, there exists an open neighborhood $G \subset A$ of $h$ such that $\sup_{\omega \in G \cap S}\phi_{\varepsilon_k} (\omega) \le \phi(h) + \delta$ for all large $k$. Hence,
\[
P^{\varepsilon_k} (A) \ge P^{\varepsilon_k} (G) \ge e^{-(\phi (h)+\delta)/\varepsilon_k} \bar{R}^{\varepsilon_k}(G).
\]
Corollary \ref{prop: LDP B bridge} implies that 
\[
 \varepsilon_k \log P^{\varepsilon_k} (A) \ge -\phi(h) - \delta - J(h) + o(1)
\]
as $k \to \infty$. Noting that $\phi(h) + J(h) = I(h)$ yields the desired lower bound. 

\underline{Step 2}. For the upper bound,  we first note that by Lemma \ref{lem: uniqueness}, $\phi_{\varepsilon}$ are uniformly lower bounded on $S$, $\phi_{\varepsilon}(\omega) \ge -M$ for all $\omega \in S$ and $\varepsilon > 0$ for some $M >0$. Let $A \subset E$ be closed. Pick any $\alpha < \infty$ and $\delta > 0$. Set $\Psi_{J}(\alpha) = \{ h : J(h) \le \alpha \} \cap A$, which is a compact subset of $E$ as $J$ is a good rate function and $A$ is closed. By Lemma \ref{lem: uniqueness} and lower semicontinuity of the function $J$, for every $h \in \Psi_{J}(\alpha)$ (which entails $h \in S$), one can find an open neighborhood $U_h$ of $h$ such that 
\[
\inf_{\omega \in \bar{U}_h} J(\omega) \ge J(h) - \delta, \quad \inf_{\omega \in \bar{U}_h \cap S} \phi_{\varepsilon_k} (\omega) \ge \phi(h) - \delta \quad \text{for large $k$},
\]
where $\bar{U}_h$ denotes the closure of $U_h$ in $E$. By compactness of $\Psi_J(\alpha)$, one can find $h_1,\dots,h_N \in \Psi_J(\alpha)$ such that $\Psi_J(\alpha) \subset \bigcup_{i=1}^N U_{h_i}$. Now, setting $F=\left(\bigcup_{i=1}^N U_{h_i}\right )^c \cap A$ (which is a closed subset of $E$), we observe that
\[
P^{\varepsilon_k}(A) = \int_A e^{-\phi_{\varepsilon_k}/\varepsilon_k} \, d\bar{R}^{\varepsilon} \le \sum_{i=1}^N e^{( \varepsilon_k \log \bar{R}^{\varepsilon_k}(\bar{U}_{h_i}) - \phi(h_i) + \delta)/\varepsilon_k}  + e^{(M+\varepsilon_k \log \bar{R}^{\varepsilon_k}(F))/\varepsilon_k}.
\]
Using Lemma \ref{lem: smooth max}, and combining Lemma \ref{lem: uniqueness} and Corollary \ref{prop: LDP B bridge}, we have 
\[
\begin{split}
\varepsilon_k \log P^{\varepsilon_k} (A) &\le \max \Big\{  \varepsilon_k \log \bar{R}^{\varepsilon_k}(\bar{U}_{h_1}) - \phi(h_1) + \delta, \dots, \varepsilon_k \log \bar{R}^{\varepsilon_k}(\bar{U}_{h_N}) - \phi(h_N) + \delta,\\
& \qquad\qquad  M+\varepsilon_k \log \bar{R}^{\varepsilon_k}(F) \Big \} + \varepsilon_k \log (N+1) \\
&\le  \max \Big\{  -\inf_{\omega \in \bar{U}_{h_1}} J(\omega) - \phi(h_1) + \delta, \dots, -\inf_{\omega \in \bar{U}_{h_N}} J(\omega) - \phi(h_N) + \delta,\\
& \qquad\qquad  M-\inf_{\omega \in F} J(\omega) \Big \} + o(1) \\
&\le \max \left \{ -I(h_1) + 2\delta,\dots,-I(h_N)+2\delta, M-\alpha \right \} +o(1) \\
&\le \max \left \{ -\inf_{h \in A} I(h) + 2\delta, M-\alpha \right \} + o(1),
\end{split}
\]
where we used $J(h) + \phi (h) = I(h)$. Since $\alpha < \infty$ and $\delta >0$ are arbitrary, we obtain the desired upper bound. 
Finally, the rate function $I$ being good follows from a similar argument to the proof of Corollary \ref{prop: LDP B bridge} (iii). This completes the proof. 
\end{proof}

\subsection{Proof of Proposition \ref{prop: two point}}

 The fact that the sequence $\{ P^{\varepsilon_k}_{st} \}_{k \in \NN}$ satisfies an LDP having a good rate function follows from Corollary \ref{thm: main2} and the contraction principle. The rate function is given by
\[
    I_{st}(x,y) = \inf_{h: h(s)=x, h(t)=y} \frac{\| h \|_H^2}{2} - \varphi(h(0)) - \psi(h(1)). 
    \]
    First, fix two endpoints $h(0)=x'$ and $h(1)=y'$ and optimize $\| h \|_H^2$ under the constraint $(h(s),h(t))=(x,y)$. The optimal $h$ is given by
\begin{equation}
h(u) = 
\begin{cases}
\left ( 1-\frac{u}{s} \right) x' + \frac{u}{s}x & \text{if} \ u \in [0,s], \\
\left ( 1-\frac{u-s}{t-s} \right) x + \frac{u-s}{t-s}y & \text{if} \ u \in [s,t], \\
\left ( 1-\frac{u-t}{1-t} \right) y + \frac{u-s}{1-t}y' &  \text{if} \ u \in [t,1],
\end{cases}
\label{eq: path}
\end{equation}
which gives $\| h \|_H^2/2 = c^{0,s}(x',x)+c^{s,t}(x,y)+c^{t,1}(y,y')$. Hence,
\[
\begin{split}
I_{st}(x,y) &= \inf_{x',y'} \big \{  c^{0,s}(x',x)+c^{s,t}(x,y)+c^{t,1}(y,y') - \varphi(x') - \psi(y') \big \} \\
&= c^{st}(x,y) + \mathcal{Q}_{s}(-\varphi)(x) + \mathcal{Q}_{1-t}(-\psi)(x) = c^{s,t}(x,y) - \varphi_s(x) - \psi_t(y). 
\end{split}
\]
The final claim follows from Theorem 7.35 in \cite{villani2009optimal} after adjusting the signs. 
\qed

\subsection{Proof of Proposition \ref{prop: langevin}}
{The EOT plan $\check{\pi}^\varepsilon$ is of the form
\[
d\check{\pi}^{\varepsilon}(x,y) = e^{(\check{\varphi}_\varepsilon(x)+\check{\psi}_\varepsilon(y)-c_{\varepsilon}(x,y))/\varepsilon} \, d(\mu_0 \otimes \mu_1)(x,y),
\]
where $(\check{\varphi}_\varepsilon,\check{\psi}_\varepsilon)$ are EOT potentials satisfying the Schr\"{o}dinger system (\ref{eq: S system}) with $c$ replaced by $c_\varepsilon$. For uniqueness, we assume without loss of generality  $\int \check{\varphi}_\varepsilon \, d\mu_0 = \int \check{\psi}_\varepsilon \, d\mu_1$. Consider the mixture distribution $Q^{\varepsilon} = \int \check{R}^{\varepsilon,xy} \, d(\mu_0 \otimes \mu_1)(x,y)$, then 
\[
\frac{d\check{P}^\varepsilon}{dQ^\varepsilon} (\omega)= \exp \left \{\frac{1}{\varepsilon} \left(\check{\varphi}_\varepsilon(\omega(0))+\check{\psi}_\varepsilon(\omega(1))-c_{\varepsilon}(\omega(0),\omega(1))\right) \right \}, \ \omega = (\omega(t))_{t \in [0,1]} \in E. 
\]
Furthermore, by Theorems 4.4.6 and 4.4.12 in \cite{stroock2008partial}, one has 
\begin{equation}
\lim_{\varepsilon \downarrow 0} c_{\varepsilon}(x,y) = \frac{|x-y|^2}{2} = c(x,y) \quad \text{uniforly over $(x,y) \in \calX \times \calY$}.
\label{eq: uniform}
\end{equation}
Hence, in view of the direct proof of Corollary \ref{thm: main2}, the desired claim follows once we verify the following:
\begin{itemize}
\item The mixture distributions $\{ Q^{\varepsilon_k} \}_{k \in \NN}$ satisfy the LDP with good rate function $J(h) = \inf_{(x,y) \in \calX \times \calY} J_{xy}(h)$;
\item As $k \to \infty$, $\check{\varphi}_{\varepsilon_k} \to \psi^c$ and $\check{\psi}_{\varepsilon_k} \to \psi$ uniformly on $\calX$ and $\calY$, respectively. 
\end{itemize}
The first item follows by establishing exponential continuity of $\{ \check{R}^{\varepsilon_k, xy} \}_{k \in \NN}$ w.r.t. $(x,y)$. To this end, we invoke the Radon-Nikodym derivative of the Langevin bridge $\check R^{\varepsilon,xy}$ against the Brownian bridge $R^{\varepsilon,xy}$:
\begin{equation}
\frac{d\check{R}^{\varepsilon,xy}}{dR^{\varepsilon,xy}}(\omega)=Z_{\varepsilon,xy}^{-1}\exp \left \{ -\frac{\varepsilon}{2} \int_0^1 \big(|\nabla V(\omega(t))|^2 -\Delta V(\omega(t)) \big) \, dt \right \}, 
\label{eq: RN}
\end{equation}
where $\Delta V$ is the Laplacian of $V$ and $Z_{\varepsilon,xy}$ is the normalizing constant. 
See Section 5 in \cite{levy1993dynamics} and the proof of Theorem 2.1 in \cite{conforti2018fluctuations}; see also Remark \ref{rem: RN} below. Heuristically, this follows from the following observation. The Langevin diffusion $X^\varepsilon$ follows the SDE
\[
dX^\varepsilon (t) = -\varepsilon \nabla V(X^\varepsilon(t)) \, dt + \sqrt{\varepsilon} \, dW(t).
\]
The Girsanov theorem yields that 
\[
\frac{d\check{R}^\varepsilon}{dR^\varepsilon}(\omega) = \exp \left \{ -\int_0^1 \nabla V(\omega(t)) \cdot d\omega(t) - \frac{\varepsilon}{2} \int_0^1 |\nabla V(\omega(t))|^2 \, dt \right \}
\]
under $R^\varepsilon$. An application of Ito's formula yields that 
\[
\int_0^1 \nabla V(\omega(t)) \cdot d\omega(t) = V(\omega(1)) - V(\omega(0))  - \frac{\varepsilon}{2} \int_0^1 \Delta V(\omega(t)) \, dt
\]
under $R^\varepsilon$. The bridge case is obtained by canceling $V(\omega(1)) - V(\omega(0))$, which is to be expected since it depends only on the endpoints.
Now, since the potential $V$ has bounded derivatives, the desired exponential continuity follows from Proposition \ref{prop: exponential}. }

{
For the second item, by the Schr\"{o}dinger system and Jensen's inequality, one has 
\[
\begin{split}
|\check{\varphi}_\varepsilon(x) - \check{\varphi}_\varepsilon (x')| &\le \sup_{y \in \calY}|c_\varepsilon(x,y) - c_\varepsilon(x',y)| \\
&\le \sup_{y \in \calY}|c(x,y) - c(x',y)| + 2\sup_{\calX \times \calY}|c_\varepsilon-c|.
\end{split}
\]
By the generalized Ascoli-Arzel\`{a} theorem (cf. Lemma 2.2 in \cite{nutz2021entropic}), the sequence of functions $\{ \check{\varphi}_{\varepsilon_k} \}_{k \in \NN}$ converges uniformly on $\calX$ along a subsequence. A similar result holds for $\check{\psi}_{\varepsilon_k}$. The rest of the proof is analogous to the second part of the proof of Lemma \ref{lem: uniqueness}. This completes the proof. }
\qed

{
\begin{remark}[Derivation of (\ref{eq: RN})]
\label{rem: RN}
Formally, the Radon-Nikodym derivative (\ref{eq: RN}) follows by reducing to the $\varepsilon=1$ case via reparameterization and the formula (25) in \cite{conforti2018fluctuations}. Indeed, the process $Y^\varepsilon(t)=X^{\varepsilon}(t)/\sqrt{\varepsilon}$ satisfies 
\[
dY^{\varepsilon}(t) = -\nabla V^\varepsilon(Y^\varepsilon(t)) dt + dW(t),
\]
where $V^{\varepsilon}(x) = V(\sqrt{\varepsilon}x)$. By the formula (25) in \cite{conforti2018fluctuations}, denoting by $Y^{\varepsilon}_{\#}\Prob$ the law of the process $Y^\varepsilon = (Y^\varepsilon (t))_{t \in [0,1]}$, one has
\[
\begin{split}
\frac{d (Y^{\varepsilon}_{\#}\Prob)^{xy}}{dR^{1,xy}} (\omega) &= Z_{xy}^{-1}\exp \left \{ -\frac{1}{2} \int_0^1 \big ( |\nabla V^\varepsilon (\omega(t))|^2 - \Delta V^\varepsilon (\omega(t)) \big) \, dt \right \} \\
&=Z_{xy}^{-1}\exp \left \{ -\frac{\varepsilon}{2} \int_0^1 \big ( |\nabla V (\sqrt{\varepsilon}\omega(t))|^2 - \Delta V (\sqrt{\varepsilon}\omega(t))\big) \, dt \right \},
\end{split}
\]
where $Z_{xy}$ is the normalizing constant. 
Now, the formula (\ref{eq: RN}) follows by a simple reparameterization. 
\end{remark}
}

\section*{Acknowledgments}
The author would like to thank the editor and two anonymous referees for their careful reading and constructive comments that helped improve the quality of this paper.

\bibliographystyle{plain}
\bibliography{ref}
\end{document}